\documentclass[11pt,reqno,twoside]{amsart}

\usepackage[
  left=1.5in,
  right=1.5in,
  top=1.5in,
  bottom=1.6in,
  footskip=0.45in,
  headheight=14pt,
  headsep=27pt
]{geometry}

\usepackage[T1]{fontenc}
\usepackage{newtxtext}
\usepackage{mathtools}
\usepackage[smallerops]{newtxmath}
\usepackage{mathrsfs}
\usepackage{bm}
\usepackage{microtype}

\usepackage{graphicx}
\usepackage{xcolor}
\usepackage{tikz}
\usepackage{booktabs}
\usepackage{enumitem}

\makeatletter
\renewcommand{\section}{%
  \@startsection{section}{1}{\z@}%
  {24pt plus 4pt minus 2pt}%
  {12pt plus 2pt minus 1pt}%
  {\normalfont\bfseries\scshape\centering}%
}

\renewcommand{\subsection}{%
  \@startsection{subsection}{2}{\z@}%
  {16pt plus 3pt minus 2pt}%
  {10pt plus 2pt minus 1pt}%
  {\normalfont\bfseries\scshape}%
}

\renewcommand{\subsubsection}{%
  \@startsection{subsubsection}{3}{\z@}%
  {14pt plus 3pt minus 2pt}%
  {8pt plus 2pt minus 1pt}%
  {\normalfont\bfseries}%
}

\renewcommand{\paragraph}{%
  \@startsection{paragraph}{4}{\z@}%
  {12pt plus 2pt minus 1pt}%
  {-1em}%
  {\normalfont\bfseries}%
}

\def\@secnumfont{\bfseries}
\makeatother

\usepackage{fancyhdr}

\newcommand{\runningtitle}{%
  \small Determinantal dynamics from
  $\mathrm{GL}_N(\mathbb{C})$ Brownian motion%
}
\newcommand{\runningauthors}{\small Z. S. Mirsajjadi}

\fancypagestyle{plain}{%
  \fancyhf{}%
  \fancyfoot[C]{\thepage}%
}

\allowdisplaybreaks
\AtBeginDocument{%
  \setlength{\abovedisplayskip}{11pt plus 3pt minus 2pt}%
  \setlength{\belowdisplayskip}{11pt plus 3pt minus 2pt}%
  \setlength{\abovedisplayshortskip}{9pt plus 3pt minus 2pt}%
  \setlength{\belowdisplayshortskip}{9pt plus 3pt minus 2pt}%
  \setlength{\jot}{4pt}%
}

\usepackage{needspace}

\usepackage[hidelinks]{hyperref}
\usepackage[
  nameinlink,
  capitalise,
  noabbrev
]{cleveref}

\hypersetup{%
 pdftitle={Infinite-Particle Determinantal Dynamics from Brownian Motion on
\texorpdfstring{$\mathrm{GL}_N(\mathbb{C})$}{GL\_N(C)}},
  pdfauthor={Zahra Sadat Mirsajjadi},
  pdfsubject={Mathematics},
  pdfkeywords={Random matrices, correlation kernels, stochastic processes}
}

\usepackage{etoolbox}
\makeatletter
\apptocmd{\thebibliography}{%
  \fontsize{10.4pt}{13.5pt}\selectfont
  \setlength{\itemsep}{5pt}%
  \setlength{\parsep}{0pt}%
}{}{}
\makeatother

\newtheorem{thm}{Theorem}[section]
\newtheorem{lem}[thm]{Lemma}
\newtheorem{prop}[thm]{Proposition}
\newtheorem{cor}[thm]{Corollary}

\theoremstyle{definition}
\newtheorem{defn}[thm]{Definition}

\theoremstyle{remark}
\newtheorem{rmk}[thm]{Remark}

\numberwithin{equation}{section}

\newcommand{\defeq}{\mathrel{\overset{\textnormal{def}}{=}}}
\newcommand{\distreq}{\mathrel{\overset{\textnormal{d}}{=}}}
\newcommand{\Ibr}[1]{\llbracket\hspace{-.2mm}#1\hspace{-.2mm}\rrbracket}
\newcommand{\LP}{\mathscr{L}\!\mathscr{P}}

\newcommand{\sbullet}{%
  \mathbin{\raisebox{0.3ex}{\scalebox{0.5}{$\bullet$}}}%
}

\title[Determinantal Dynamics from $\mathrm{GL}_N(\mathbb{C})$ Brownian motion]
{Infinite-Particle Determinantal Dynamics from Brownian Motion on
$\mathrm{GL}_N(\mathbb{C})$}

\author[Z.S.Mirsajjadi]{Zahra Sadat Mirsajjadi}

\renewenvironment{abstract}
{%
  \begin{center}
    \normalsize\bfseries Abstract
  \end{center}
  \vspace{1pt}%
  \begin{quotation}
    \small
}
{%
  \end{quotation}
}

\begin{document}

\maketitle

\begin{abstract}
We study the edge scaling limit of the logarithms of the squared
singular values of multiplicative Brownian motion on the general linear group 
$\mathrm{GL}_N(\mathbb{C})$ with general initial conditions. 
Building on the results of 
\cite{AssiotisMirsajjadi2026}, we prove that, for every initial condition in the enhanced state space, the
limiting dynamics are determinantal and derive an explicit double-contour formula for their extended multitime correlation kernel. The dependence on the initial state, including information not determined by the particle coordinates,
is encoded by
the associated Laguerre--P\'olya entire function. 
We further extend the construction of solutions to the corresponding infinite-dimensional system of singular stochastic
differential equations to arbitrary initial configurations,
including those with coinciding coordinates.
\end{abstract}

\section{Introduction}
\noindent
Multiplicative Brownian motion on the general linear group
$\mathrm{GL}_N(\mathbb{C})$
\cite{NRW,Biane-freeBm,Jones-OConnell,IsotropicBm} may be viewed as a
multiplicative analogue of Hermitian matrix Brownian motion. 
Its squared
singular values form an interacting particle system with logarithmic
interaction potential \cite{AssiotisMirsajjadi2026}, whose logarithms
admit a representation in terms of non-colliding Brownian motions with
equally spaced drifts
\cite{Jones-OConnell,Katori-kernel,Takahashi-Katori}; see also
\cite{AssiotisMirsajjadi2026}.
Under appropriate edge scaling, these dynamics give rise in the
large-$N$ limit to an infinite-dimensional Feller process
\cite{AssiotisMirsajjadi2026}. 
When the matrix process is started from the identity, the limiting
dynamics also arises in the asymptotic study of products of random
matrices under a balanced scaling of the number of factors and the
matrix dimensions
\cite{Ahn,Akemann2019integrable,AkemannUniversality,LWW}.

For the specific identity initial condition, Ahn \cite{Ahn}
established convergence of the joint Laplace transforms and correlation
functions of the edge-scaled logarithms of the squared singular values
and, in particular, obtained an explicit formula for the extended
correlation kernel of the limiting process; see also \cite{Rahman} for a related fixed-time kernel representation.  
Ahn further proved
convergence of the extremal paths on compact time intervals using the
machinery developed in \cite{CorwinHammond}.

A substantially more general path-space convergence result was recently
established in \cite{AssiotisMirsajjadi2026} using the framework
developed in \cite{AssiotisMirsajjadi2024}, which builds on the method of intertwiners originally introduced  by Borodin and
Olshanski \cite{MarkovProc-pathSpace-GT,MarkovDynam-ThomaCone,OlshanskiLectureNotes}. 
For arbitrary drift and general initial conditions, Assiotis and
Mirsajjadi construct the path-space edge scaling limit of the squared
singular values and establish the Feller property of the limiting
dynamics. They also obtain a stochastic partial differential equation
governing the evolution of the limiting reverse characteristic
polynomial. For strictly positive and strictly ordered initial
coordinates, they further establish a Gibbs resampling property for the
limiting line ensemble and derive the associated infinite-dimensional
system of stochastic differential equations (ISDE) with logarithmic
interactions.

These results, in particular, extend both the path-space convergence and the
probabilistic description of the limiting dynamics beyond the identity
initial condition. The corresponding correlation kernels are not, however, derived in
\cite{AssiotisMirsajjadi2026}.  The purpose of this paper is to establish the determinantal structure of the limiting process for general initial conditions and to derive explicit formulas for its fixed-time and extended multitime correlation kernels. 
 The limiting kernels encode the initial state through its associated
Laguerre--P\'olya function and yield, for the identity initial
condition, alternative contour-integral formulas for the correlation
functions of the limiting dynamics constructed in \cite{Ahn}.

Determinantal correlation kernels have long played a central role in
the study of non-colliding particle systems and dynamical limits in
random matrix theory; see, among others,
\cite{KarlinMcGregor,EynardMehta1998,Soshnikov2000,JohanssonDet,BorodinDet,
ForresterBook,KatoriTanemuraNoncollidingBM,KatoriTanemuraMarkov} and the references therein. 
Classical examples include the extended correlation kernels associated with Dyson Brownian motion and its bulk and edge scaling limits 
\cite{NagaoForrester1998,PrahoferSpohn,TracyWidomDyson}.

In
finite dimensions, Dyson Brownian motion at $\beta=2$ and related
non-colliding diffusions admit explicit fixed-time and multitime
determinantal descriptions for general deterministic initial
configurations
\cite{KatoriTanemuraMarkov,KatoriTanemuraComplexBM}. More generally,
Assiotis \cite{ExactSolution} obtained explicit space-time correlation
kernels from arbitrary deterministic initial conditions for a broad
class of finite-dimensional non-colliding diffusions associated with
classical random matrix ensembles.

Infinite-particle limits of the
classical Dyson and squared-Bessel dynamics  give rise to several 
fundamental dynamical point processes of random matrix theory,
including the extended $\mathrm{sine}$, $\mathrm{Airy}$ and $\mathrm{Bessel}$ processes; see
\cite{KatoriTanemuraMarkov,OsadaTanemuraStrongMarkov}
and the references therein. 
These classical processes arise primarily as equilibrium dynamics or from distinguished initial conditions.

Particularly relevant to the present work are the non-equilibrium
infinite-particle results of Katori and Tanemura
\cite{KatoriTanemura1,KatoriTanemura2} and the very recent work of Assiotis and Li 
\cite{AssiotisLi2026}. For suitable classes of infinite initial point
configurations, Katori and Tanemura
\cite{KatoriTanemura1,KatoriTanemura2} constructed determinantal
Dyson and non-colliding squared Bessel dynamics whose multitime
correlation kernels retain the dependence on the initial configuration
through associated entire functions; see also
\cite{KatoriTanemuraMarkov,KatoriTanemuraComplexBM}. In these
constructions, the corresponding entire functions are determined by the
initial point configuration via Weierstrass canonical products.

Assiotis and Li \cite{AssiotisLi2026} extended the non-equilibrium
determinantal results of Katori and Tanemura for infinite-dimensional
Dyson Brownian motion.
 At $\beta=2$, they constructed determinantal
processes on an extended space of initial data and obtained explicit
extended correlation kernels. Additional parameters recording
information not determined by the limiting particle configuration enter
these kernels through an associated Laguerre--P\'olya entire function. 
They also used these kernels to prove convergence in finite-dimensional distributions for a broad class of initial data and, for certain symmetric initial configurations, a rescaled long-time limit to the stationary extended $\mathrm{sine}$ process.

The present work establishes a corresponding determinantal structure
for the multiplicative dynamics considered here.
Our results provide
an explicit extended multitime correlation kernel for every initial
state in the full enhanced state space of the limiting Feller process, allowing in particular for coincident and degenerate configurations. 
 The associated Laguerre--P\'olya function encodes the initial state in its entirety. 
Consequently, identical particle configurations may give rise to
different correlation kernels through the additional parameter in the
corresponding enhanced states.

The results of the present work further complete the description of
the limiting dynamics by extending the Gibbs resampling property and
the associated ISDE formulation of
\cite{AssiotisMirsajjadi2026} to arbitrary initial states, thereby covering degenerate configurations
beyond the scope of the earlier results. 
This yields a construction of solutions to an infinite-dimensional
singular log-interacting SDE from all admissible deterministic initial
data, including configurations with coinciding particle coordinates, where 
the singular interaction at time zero is interpreted as an improper
time integral. To the best of our knowledge, such a construction is
novel. 
See, among others, 
\cite{Tsai,OsadaPTRF,OsadaAOP,OsadaTanemura,AssiotisMirsajjadi2024,
AssiotisMirsajjadi2026,AssiotisLi2026} for results on ISDEs arising
from random matrix dynamics.

\subsection{Main results}

\noindent
For any $N\in\mathbb{N}$, we denote by
$\mathrm{M}_N(\mathbb{C})$ the
space of all $N\times N$ complex matrices and by
$\mathrm{GL}_N(\mathbb{C})$ its group of invertible elements. The
drifted multiplicative Brownian motion on
$\mathrm{GL}_N(\mathbb{C})$, with drift
$\theta\in\mathbb{R}$, is defined as the unique strong solution of the
following matrix-valued SDE:
\begin{align}\label{GL_NBmSDE}
\mathrm{d}\mathsf{Y}_N^\theta(t)
=
\mathsf{Y}_N^\theta(t)\,\mathrm{d}\mathbf{W}_N(t)
+
\theta\mathsf{Y}_N^\theta(t)\,\mathrm{d}t,
\quad \mathsf{Y}_N^\theta(0)\in \mathrm{GL}_N(\mathbb{C}),
\end{align}
where 
$(\mathbf{W}_N(t))_{t\ge 0}$ is an $\mathrm{M}_N(\mathbb{C})$-valued additive
complex Brownian motion with independent standard real and imaginary
parts.

Note that, 
when the initial matrix is invertible, the solution of
\eqref{GL_NBmSDE} is $\mathrm{GL}_N(\mathbb{C})$-valued. Nevertheless, the  matrix SDE  admits a
unique strong solution from every initial condition in
$\mathrm{M}_N(\mathbb{C})$ \cite{IkedaWatanabe,RevuzYor}; we shall use this natural extension when
considering singular initial matrices.

We write 
$\mathbb{R}_+\defeq[0,\infty)$ and use the notation
$\llbracket K\rrbracket\defeq\{1,\ldots,K\}$ for $K\in\mathbb{N}$.

For each $N\in\mathbb{N}$, we define the Weyl chambers
\begin{align*}
\mathbb{W}_{N,+}
&\defeq
\left\{
\bm{x}=(x_i)_{i=1}^N\in\mathbb{R}_+^N:
x_1\ge x_2\ge \cdots\ge x_N
\right\},\\
\mathbb{W}_N
&\defeq
\left\{
\bm{x}=(x_i)_{i=1}^N\in\mathbb{R}^N:
x_1\ge x_2\ge \cdots\ge x_N
\right\},
\end{align*}
with their Euclidean interiors denoted by $\mathbb{W}_{N,+}^\circ$ and $\mathbb{W}_N^\circ$ respectively. We also define the following  infinite chamber, endowed with the topology of coordinate-wise convergence,
\begin{align*}
\mathbb{W}_{\infty,+}
&\defeq
\left\{
\bm{x}=(x_i)_{i\in\mathbb{N}}\in\mathbb{R}_+^{\mathbb{N}}:
x_1\ge x_2\ge\cdots
\right\}.
\end{align*}

For $N\in\mathbb{N}$ and $\theta\in\mathbb{R}$, we denote by 
$\left(\mathsf{y}^{(N)}(t)\right)_{t\ge0}=\big(\mathsf{y}_i^{(N)}(t)\big)_{i\in\Ibr{N},\,t\ge0}$
the squared 
singular values of $(\mathsf{Y}_N(t/4))_{t\ge0}$ arranged in decreasing order.  
Here and throughout the paper, we suppress the drift parameter $\theta$ from the notation for simplicity. 

It is straightforward to verify that $(\mathsf{y}^{(N)}(t))_{t\ge0}$ satisfies the following log-interacting SDE 
\begin{align}\label{singularValues-SDE}
\mathrm{d}\mathsf{y}_i^{(N)}(t)
={}&
\mathsf{y}_i^{(N)}(t)\,\mathrm{d}\mathsf{w}_i(t)
+
\frac12\mathsf{y}_i^{(N)}(t)
\left(
1+\theta
+
\sum_{j\in\Ibr{N}\setminus\{i\}}
\frac{
\mathsf{y}_i^{(N)}(t)+\mathsf{y}_j^{(N)}(t)
}{
\mathsf{y}_i^{(N)}(t)-\mathsf{y}_j^{(N)}(t)
}
\right)\mathrm{d}t,
\quad i\in\Ibr{N},
\end{align}
with the $(\mathsf{w}_i)_{i=1}^N$ being independent standard Brownian motions. 
Moreover, the SDE~\eqref{singularValues-SDE} 
has a unique strong solution from every
initial condition in $\mathbb{W}_{N,+}$ \cite{GraczykMalecki2013,GraczykMalecki2014}, and its
transition semigroup is Feller; see
\cite{AssiotisMirsajjadi2026}.

We define the logarithms of the squared-singular-values 
$\left(\mathfrak{y}^{(N)}(t)\right)_{t\ge0}=\big(\mathfrak{y}_i^{(N)}(t)\big)_{i\in\Ibr{N},\,t\ge0}$ by 
\begin{align*}
    \mathfrak{y}_i^{(N)}(t)\defeq\log \mathsf{y}_i^{(N)}(t), \quad i\in\Ibr{N}, \quad t\ge 0,
\end{align*}
where, here and throughout, we adopt the convention $\log 0=-\infty$. If
$\mathsf{y}^{(N)}(0)\in\mathbb{W}_{N,+}^\circ$, or equivalently
$\mathfrak{y}^{(N)}(0)\in\mathbb{W}_N^\circ$, It\^o's formula gives  
\begin{align}\label{logSingularValuesSDE}
\mathrm{d} \mathfrak{y}_i^{(N)}(t) 
= \mathrm{d} \mathsf{w}_i(t)
+ 
\frac{\theta}{2}\mathrm{d}t
+
 \frac{1}{2}\sum_{j\in\Ibr{N}\setminus\{i\}} \coth\left(\frac{1}{2}\left(\mathfrak{y}_i^{(N)}(t) - \mathfrak{y}_j^{(N)}(t)\right)\right) \mathrm{d}t,
\quad i\in\Ibr{N}.
\end{align}
For $\theta=0$, \eqref{logSingularValuesSDE} is a particular radial
Heckman--Opdam process; see \cite{HeckmanOpdam}.

Moreover, when
$\mathfrak{y}^{(N)}(0)\in\mathbb{W}_N^\circ$, the process
$\mathfrak{y}^{(N)}$ has the following probabilistic interpretation
(see \cite{Jones-OConnell,BBO}):
    it
    is equal in distribution  to a collection of
$N$ independent drifted Brownian motions with equally spaced drifts:
    \begin{align*}
     \frac{\theta+N-1}{2}, 
     \frac{\theta+N-3}{2},\ldots, \frac{\theta+3-N}{2}, \frac{\theta+1-N}{2},
    \end{align*}
      conditioned never to intersect; see also \cite{Assiotis-complexMatrixBm}.

 Figure~\ref{fig-svGL_NBm-nrsc} illustrates a  simulation of the process $\mathsf{y}^{(N)}$  when $\theta=0$ with  $N=30$ particles starting from $(1,\dots,1)$, as well as the corresponding log-transformed dynamics $\mathfrak{y}^{(N)}$.

\begin{figure}[htbp]
    \centering
    \footnotesize

    \begin{minipage}[c]{0.48\textwidth}
        \centering
        \includegraphics[width=0.95\linewidth]{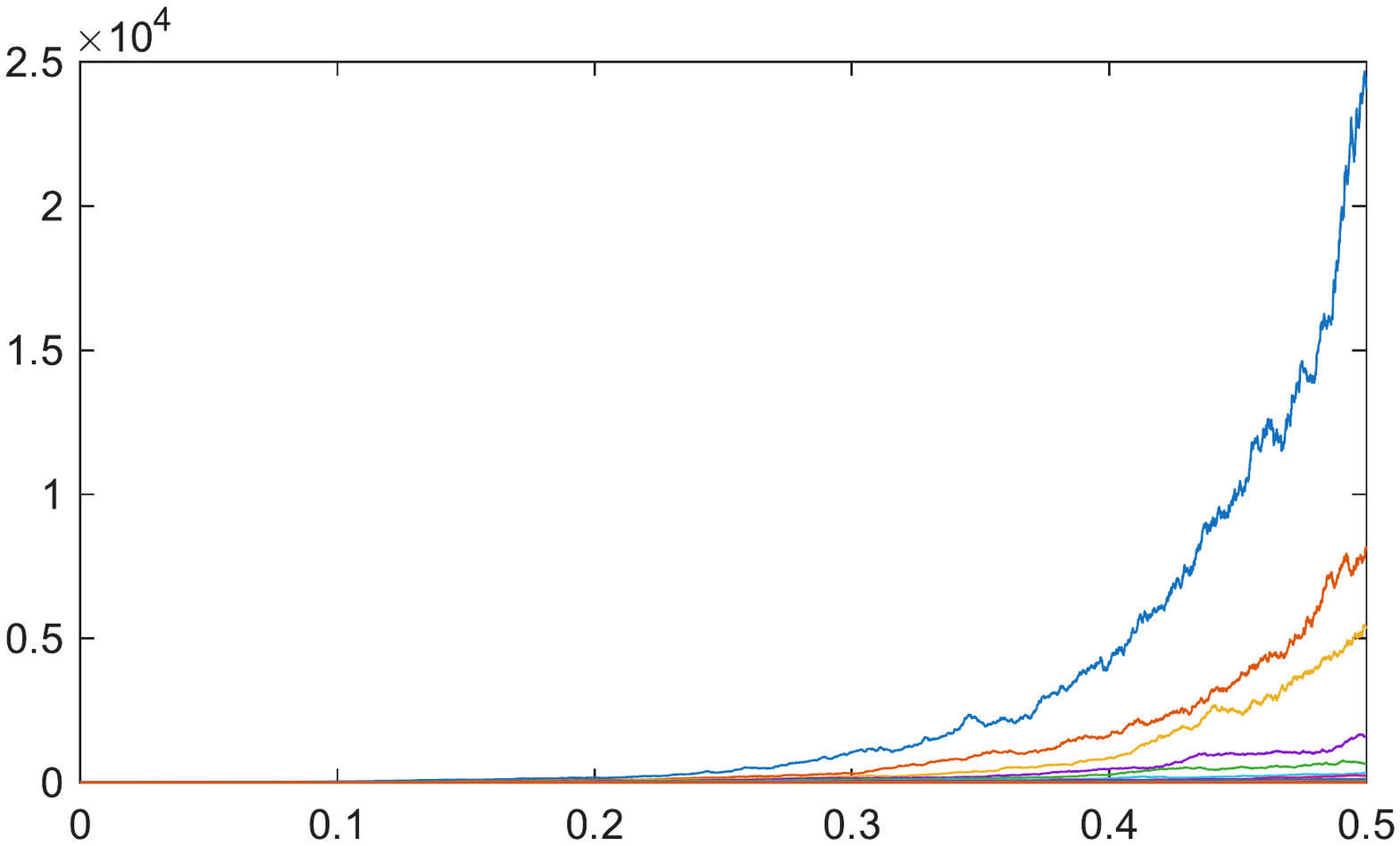}
    \end{minipage}
    \hspace{0.02\textwidth}
    \begin{minipage}[c]{0.48\textwidth}
        \centering
        \vspace{2mm}\includegraphics[width=0.96\linewidth]{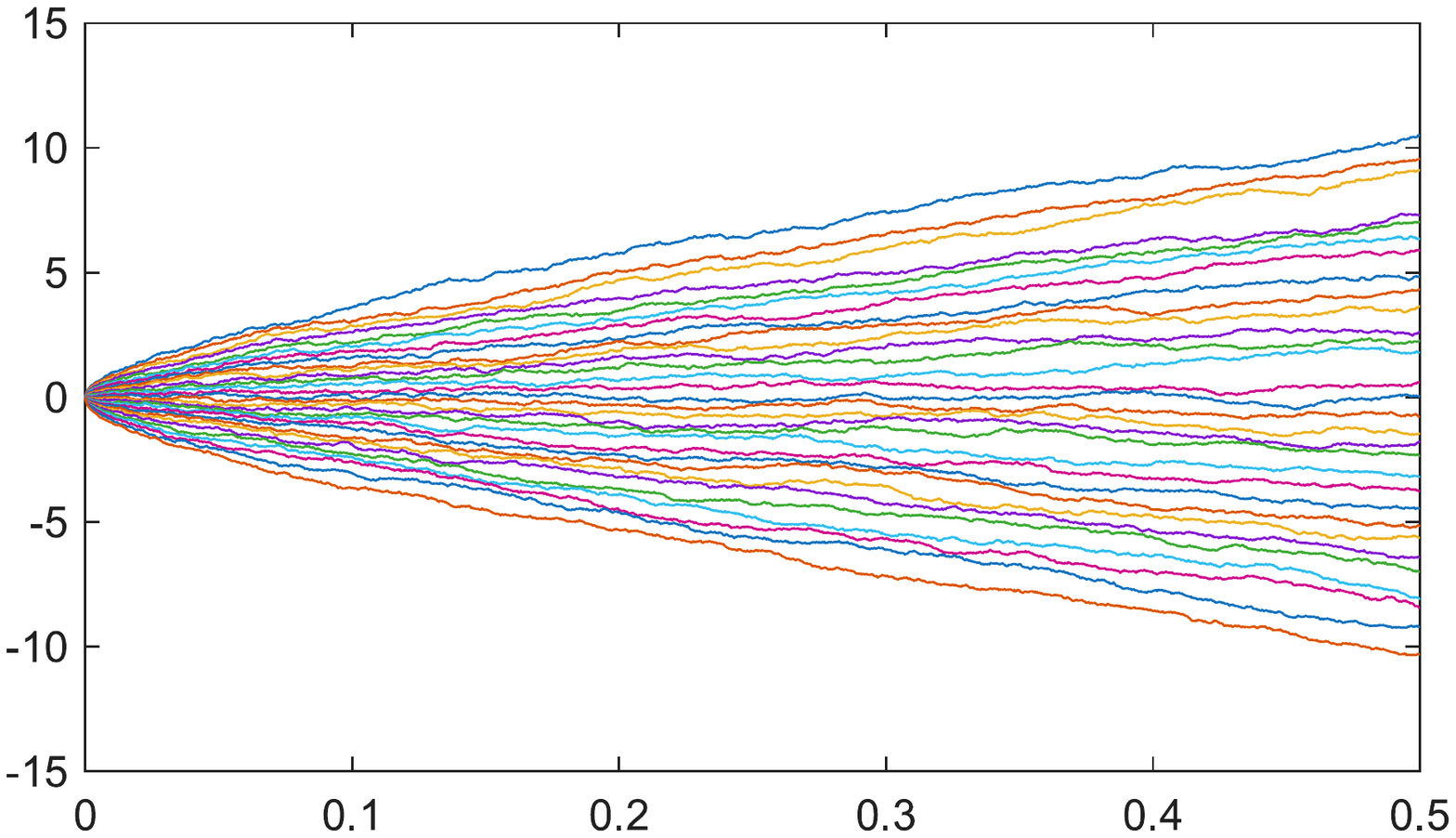}
    \end{minipage}
\caption[]{
Simulation of the process $\mathsf{y}^{(N)}$  when $\theta=0$ with   $N=30$ particles starting from $(1,\dots,1)$ (left), along with the log-transformed process $\mathfrak{y}^{(N)}$ (right).
}
\label{fig-svGL_NBm-nrsc}
\end{figure}

Let us now define the rescaled logarithmic dynamics
$\left(\mathfrak{x}^{(N)}(t)\right)_{t\ge0}
=
\bigl(\mathfrak{x}_i^{(N)}(t)\bigr)_
{i\in\Ibr{N},\,t\ge0}$ by
\begin{align}\label{rscSingularValues}
\mathfrak{x}_i^{(N)}(t)
\defeq
\mathfrak{y}_i^{(N)}(t)-\frac{N}{2}t,
\qquad
i\in\Ibr{N},\quad t\ge0.
\end{align}
We also define the corresponding exponential coordinates by
\begin{align}\label{rscExponentialCoordinates}
\mathsf{x}_i^{(N)}(t)
\defeq
\exp\bigl(\mathfrak{x}_i^{(N)}(t)\bigr),
\qquad
i\in\Ibr{N},\quad t\ge0,
\end{align}
equivalently, 
\begin{align*}
\mathsf{x}_i^{(N)}(t)
=
\mathrm{e}^{-Nt/2}\mathsf{y}_i^{(N)}(t),
\qquad
i\in\Ibr{N},\quad t\ge0.
\end{align*}
An application of It\^o's formula gives the SDE satisfied by $\left(\mathsf{x}^{(N)}(t)\right)_{t\ge0}$: 
\begin{align*}
\mathrm{d}\mathsf{x}_i^{(N)}(t)
={}&
\mathsf{x}_i^{(N)}(t)\,
\mathrm{d}\mathsf{w}_i(t)
+
\frac{\theta}{2}
\mathsf{x}_i^{(N)}(t)\,\mathrm{d}t
+
\sum_{j\in\Ibr{N}\setminus\{i\}}
\frac{
\mathsf{x}_i^{(N)}(t)\mathsf{x}_j^{(N)}(t)
}{
\mathsf{x}_i^{(N)}(t)-\mathsf{x}_j^{(N)}(t)
}\,\mathrm{d}t,
\qquad i\in\Ibr{N}.
\end{align*}
Figure~\ref{fig-rscSVGL_NBmN50} shows a  simulation of the process $\mathsf{x}^{(N)}$  when $\theta=0$ with $N=50$ particles starting from $(1,\dots,1)$, as well as the corresponding logarithmic dynamics $\mathfrak{x}^{(N)}$.

\begin{figure}[htbp] \centering
    \begin{minipage}[c]{0.48\textwidth}
        \centering
\includegraphics[width=0.95\linewidth]{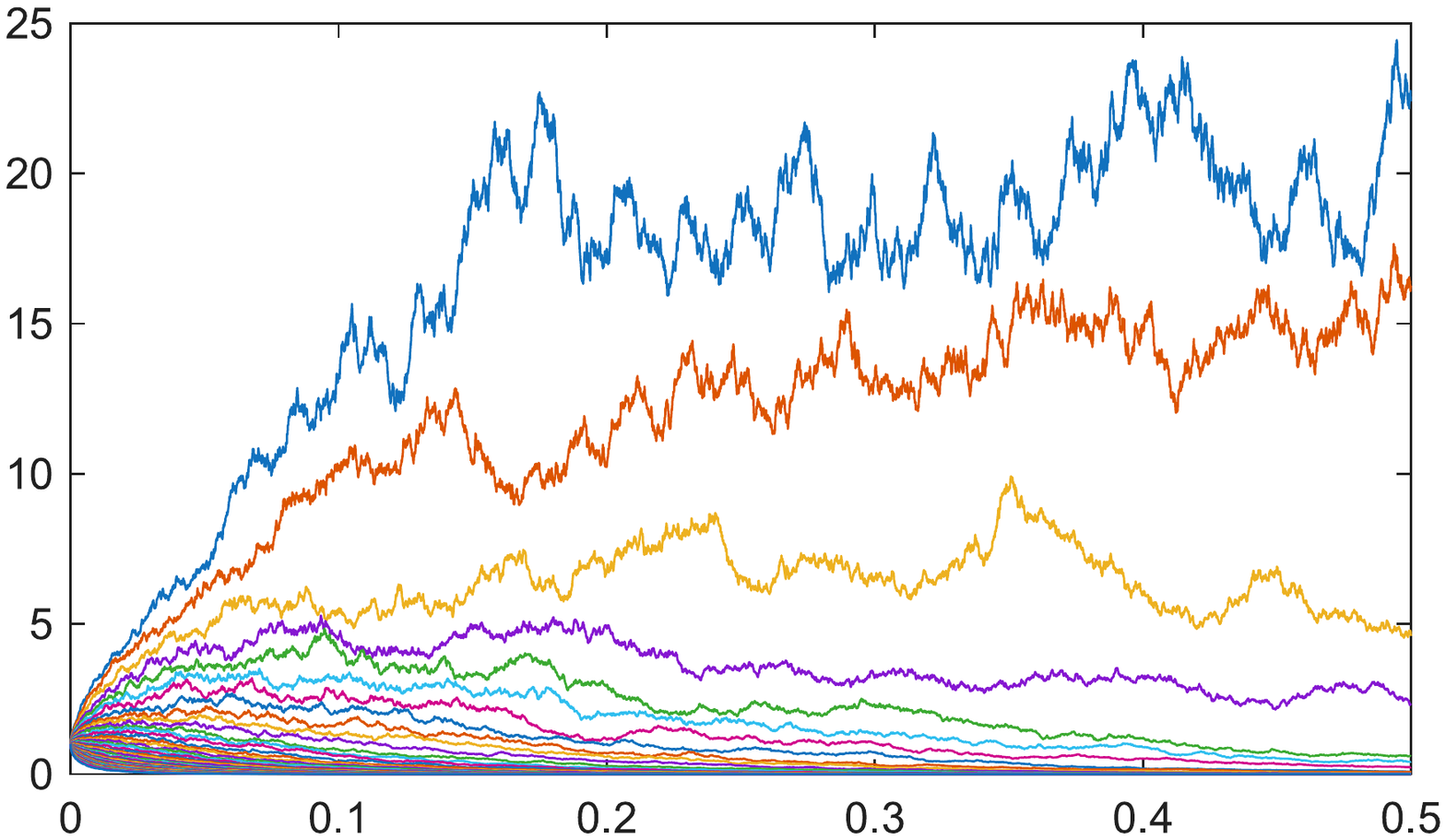}
    \end{minipage}
\hspace{0.02\textwidth}
    \begin{minipage}[c]{0.48\textwidth}
        \centering
        \includegraphics[width=0.95\linewidth]{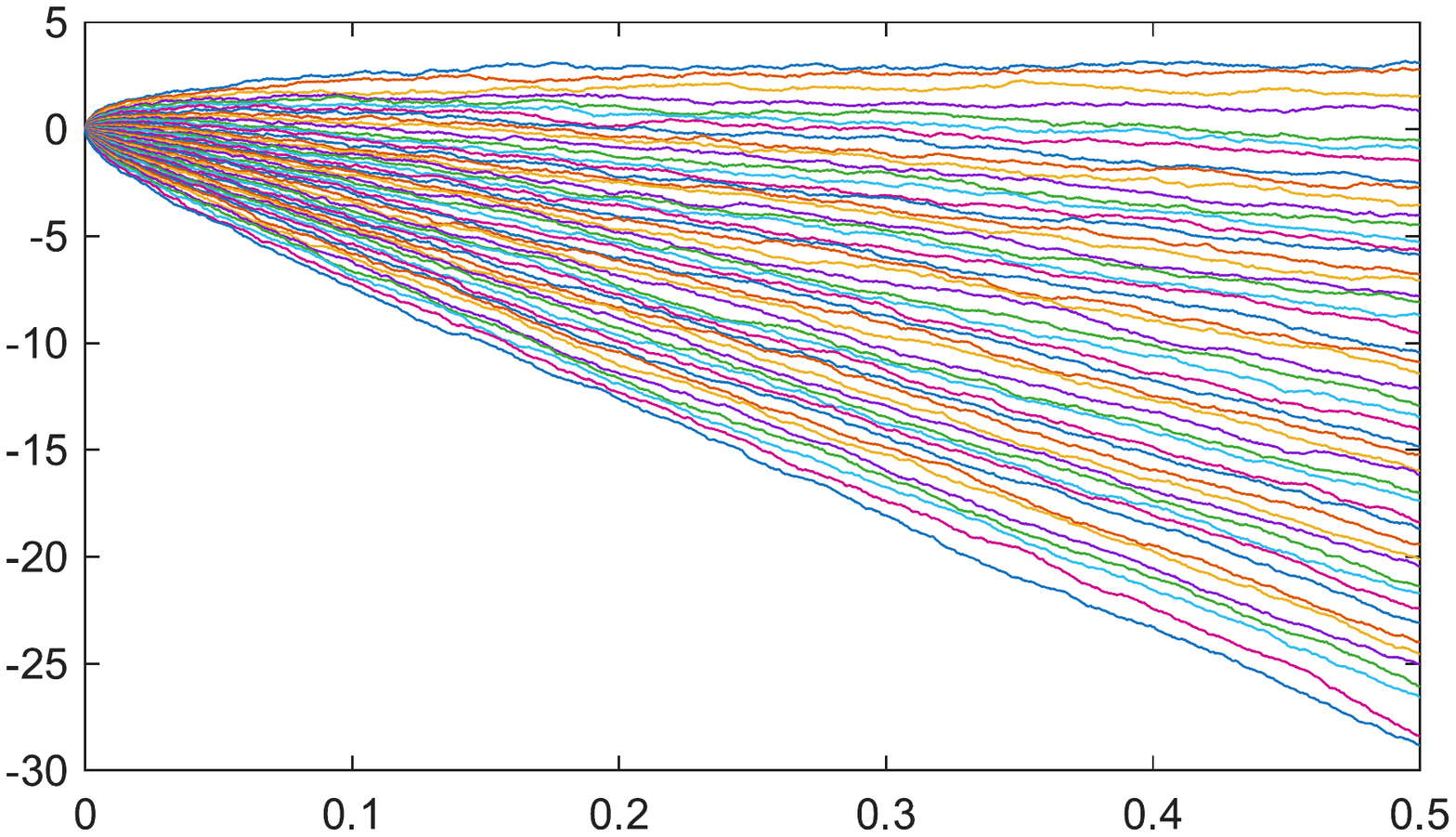}
    \end{minipage}
\caption[]{Simulation of  the process  $\mathsf{x}^{(N)}$ defined in \eqref{rscExponentialCoordinates}, when $\theta=0$  with $N=50$ particles started at $(1,\dots,1)$  (left) and its log-transformed dynamics $\mathfrak{x}^{(N)}$ (right).}
    \label{fig-rscSVGL_NBmN50}
\end{figure}

We proceed to introduce some further notation needed to state our main results.

\begin{defn}
We define
\begin{align*}
\Upsilon
\defeq
\bigg\{
\upsilon=(\bm{x},\gamma)
\in
\mathbb{W}_{\infty,+}\times\mathbb{R}_+:
\sum_{i\in\mathbb{N}}x_i\le\gamma
\bigg\},
\end{align*}
equipped with the subspace topology inherited from
$\mathbb{W}_{\infty,+}\times\mathbb{R}_+$.

For any $\upsilon=(\bm{x},\gamma)\in\Upsilon$, define the associated
entire function by the Hadamard product
\begin{align}\label{varphiUpsilon}
\varphi_{\upsilon}(z)
\defeq
\exp\bigg(
-\Big(
\gamma-\sum_{i\in\mathbb{N}}x_i
\Big)z
\bigg)
\prod_{i\in\mathbb{N}}
\left(
1-x_i z
\right),
\qquad
z\in\mathbb{C}.
\end{align}
Let $\LP_+$ denote the (non-negative) Laguerre--P\'olya class defined as
\begin{align*}
\LP_+
\defeq
\left\{
\varphi_{\upsilon}:\upsilon\in\Upsilon
\right\},
\end{align*}
equipped with the topology of uniform convergence on compact subsets
of $\mathbb{C}$.
\end{defn}

Note that $\Upsilon$ is a locally compact Polish space.
Note also that the infinite product in
\eqref{varphiUpsilon} converges locally uniformly on $\mathbb{C}$,  
and therefore defines an
entire function. Moreover, $\LP_+$ is the 
Laguerre--P\'olya class of entire functions arising as locally uniform limits of
polynomials with non-negative zeros, each normalized to equal $1$ at the origin. In addition, the map
$\upsilon\longmapsto\varphi_\upsilon$ is a homeomorphism from
$\Upsilon$ onto $\LP_+$; see \cite{Assiotis-RandomEntireFunctions}.

As we shall see shortly, the space $\Upsilon$ arises naturally as the
extended state space for the scaling limit of the rescaled
squared-singular-value processes constructed in
\cite{AssiotisMirsajjadi2026}.

For a Polish space $\mathcal{X}$, we denote by 
$\mathcal{C}(\mathbb{R}_+;\mathcal{X})$  the space of continuous
functions from $\mathbb{R}_+$ to $\mathcal{X}$, endowed with the
topology of uniform convergence on compact time intervals. This is
again a Polish space; see \cite{Kallenberg}. 
Recall that a Markov semigroup on a locally compact Polish space
$\mathcal{X}$ is Feller if it preserves
$\mathcal{C}_0(\mathcal{X})$, the space of continuous functions on
$\mathcal{X}$ vanishing at infinity, and is strongly continuous on
this space. A Feller diffusion is a Feller process with continuous
sample paths; see \cite{EthierKurtz}.

The space $\Upsilon$ is in fact identified with the Feller
boundary of the projective system underlying the finite-dimensional
dynamics $\mathsf{x}^{(N)}$ and serves as the state space of the limiting Feller
diffusion; see \cite{AssiotisMirsajjadi2024,AssiotisMirsajjadi2026}
and the references therein.

Throughout the paper, finite-dimensional vectors are canonically embedded
into the corresponding infinite-dimensional state spaces by setting the
additional coordinates equal to $0$.  
The intended convention will always
be clear from the context.

 We write $\overset{\mathrm{d}}{\longrightarrow}$ for
convergence in distribution. The following path-space convergence
result is established in \cite{AssiotisMirsajjadi2026}.

\begin{thm}[\cite{AssiotisMirsajjadi2026}]
\label{thm-pathSpaceConv}
Let $\theta\in\mathbb{R}$ be arbitrary. There exists a unique Feller diffusion
$\mathsf{X}(\sbullet)
=
\left(
\left(
\mathsf{x}_i(\sbullet)
\right)_{i\in\mathbb{N}},
\bm{\gamma}(\sbullet)
\right)$
on $\Upsilon$ 
such that if $\mathsf{X}(0)=\upsilon\in\Upsilon$ and 
\begin{align}\label{IC-conv-exp}
\left(
\left(
N^{-1}\mathsf{x}_i^{(N)}(0)
\right)_{i\in\mathbb{N}},
N^{-1}\sum_{i=1}^N \mathsf{x}_i^{(N)}(0)
\right)
\xrightarrow{N\to\infty}
\upsilon,
\qquad\textnormal{ in }\,\Upsilon,
\end{align}
then, 
\begin{align}\label{pathSpaceConv-exp}
\left(
\left(
N^{-1}\mathsf{x}_i^{(N)}(\sbullet)
\right)_{i\in\mathbb{N}},
N^{-1}\sum_{i=1}^N
\mathsf{x}_i^{(N)}(\sbullet)
\right)
\xrightarrow[N\to\infty]{\,\mathrm{d}\,}
\left(
\left(
\mathsf{x}_i(\sbullet)
\right)_{i\in\mathbb{N}},
\bm{\gamma}(\sbullet)
\right), \qquad \textnormal{ in }\, \mathcal{C}(\mathbb{R}_+;\Upsilon).
\end{align}

\end{thm}

When $\theta=0$ and
$\mathsf{x}^{(N)}(0)=(1,\ldots,1)$ for every $N$, the logarithmic
coordinate process 
$\bigl(\log\mathsf{x}_i(t)\bigr)_{i\in\mathbb N,\,t>0}$
of the resulting limit coincides in law with the line ensemble
constructed by Ahn \cite{Ahn}.

Figure~\ref{rscSVofGL_NBm-X_0=1} illustrates the finite-$N$ rescaled
dynamics $N^{-1}\mathsf{x}^{(N)}$ and their logarithmic coordinates for
$N=70$, $\theta=0$, and
$\mathsf{x}^{(N)}(0)=(1,\ldots,1)$. See also
Figure~\ref{rscSVofGL_NBm-X_0=rand} for the corresponding processes
with $N=70$ and $\theta=-2$, started from a random initial
configuration.

\begin{figure}[h!]
 \centering
 \begin{minipage}[c]{0.48\textwidth}
\centering
\includegraphics[width=0.95\linewidth]{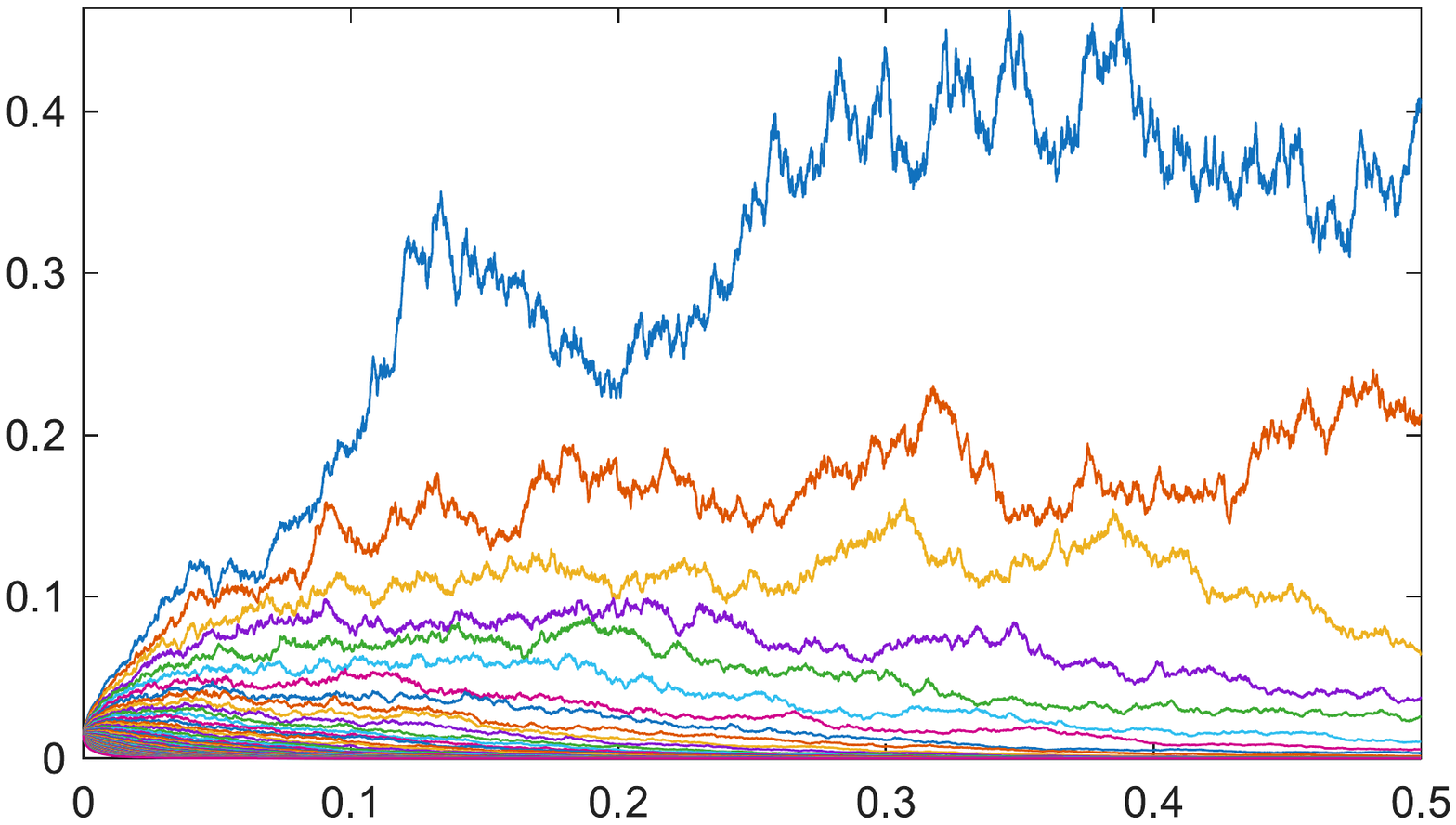}
 \end{minipage}
 \hspace{0.02\textwidth}
  \begin{minipage}[c]{0.48\textwidth}
\centering
\includegraphics[width=0.95\linewidth]{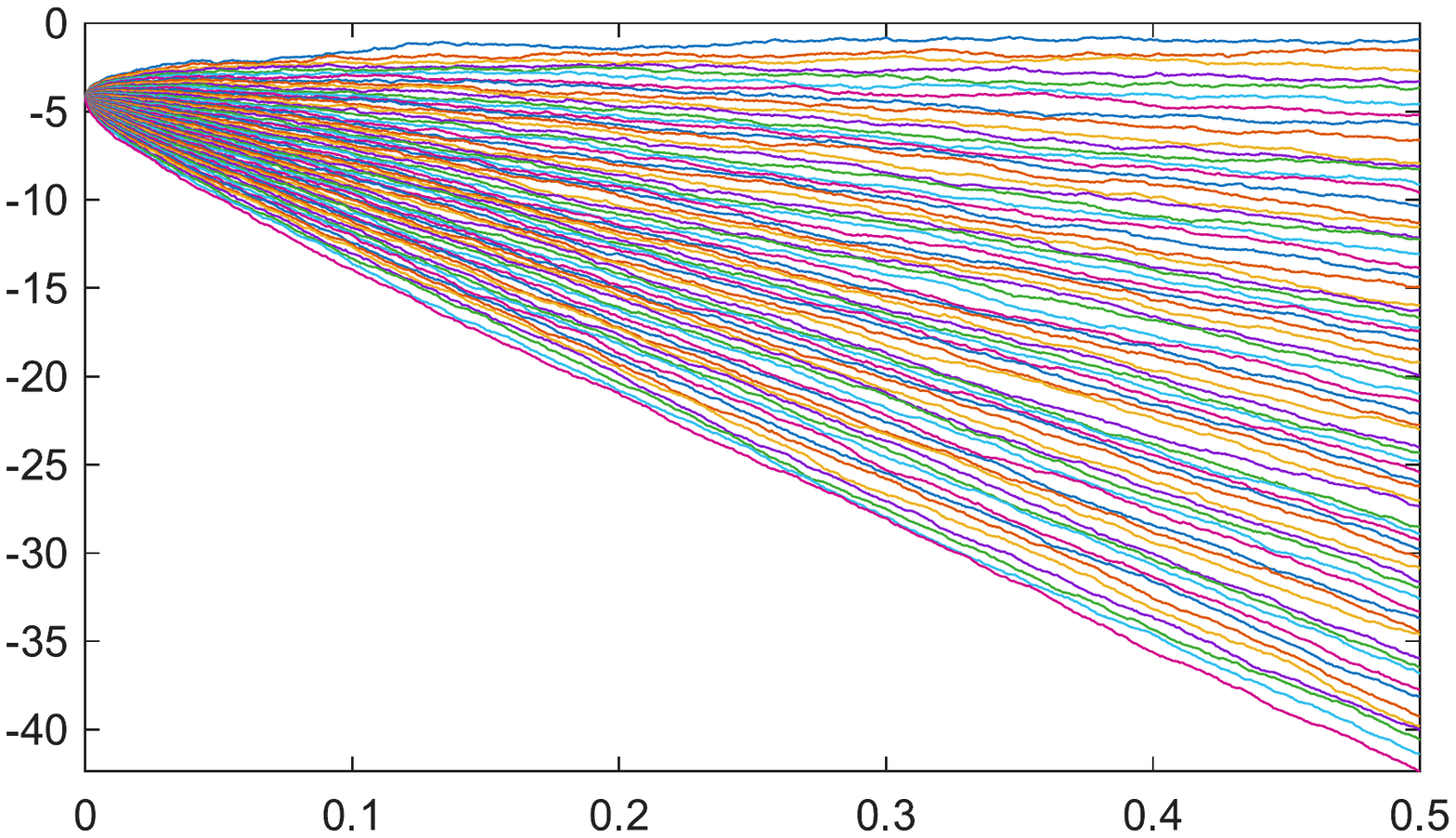}
 \end{minipage}
\caption[]{Simulation of the process $N^{-1}\mathsf{x}^{(N)}$ with $\theta = 0$   and $N = 70$ particles started at $(N^{-1},\dots,N^{-1})$ (left), along with the log-transformed dynamics (right).} 
\label{rscSVofGL_NBm-X_0=1}
\end{figure}

We now turn to the determinantal structure of the limiting logarithmic  dynamics and
first briefly recall the relevant terminology. Let
$\mathsf{Conf}(\mathbb{R})$ denote the space of locally finite point
configurations (equivalently, the space of
locally finite $\mathbb{Z}_+\cup\{\infty\}$-valued Radon measures)  on $\mathbb{R}$ 
 endowed with the vague topology. A point
process on $\mathbb{R}$, namely a probability measure on $\mathsf{Conf}(\mathbb{R})$, is said to be determinantal with correlation
kernel $\mathcal{K}$, if its
$m$-point correlation functions  with respect to Lebesgue measure satisfy
\begin{align*}
\rho_m(x_1,\ldots,x_m)
=
\det\left(
\mathcal{K}(x_i,x_j)
\right)_{i,j=1}^{m},
\quad 
m\in\mathbb{N}.
\end{align*}
For a finite collection of observation times, the analogous
space-time kernel is called an extended correlation kernel. 
Let
$\mathcal{T}=\{t_1,\ldots,t_m\}$ be a finite set of times.
A point process on $\mathcal{T}\times\mathbb{R}$ is said to be
determinantal with extended correlation kernel $\mathcal{K}$ if its
$n$-point space-time correlation functions, with respect to counting
measure on $\mathcal{T}$ and Lebesgue measure on $\mathbb{R}$, satisfy
\begin{align*}
\rho_n\bigl(
(s_1,x_1),\ldots,(s_n,x_n)
\bigr)
=
\det\left(
\mathcal{K}(s_i,x_i;s_j,x_j)
\right)_{i,j=1}^{n},
\quad 
n\in\mathbb{N},
\end{align*}
for $s_1,\ldots,s_n\in\mathcal{T}$ and
$x_1,\ldots,x_n\in\mathbb{R}$. See
\cite{Lenard,Soshnikov2000,JohanssonDet,BorodinDet} for further
details.

Throughout, following the standard convention, correlation functions
are defined with respect to Lebesgue measure in space and, in the
multitime setting, its product with counting measure on the finite set
of observation times.

\begin{figure}[h]
 \centering
 \begin{minipage}[c]{0.48\textwidth}
\centering
\includegraphics[width=0.95\linewidth]{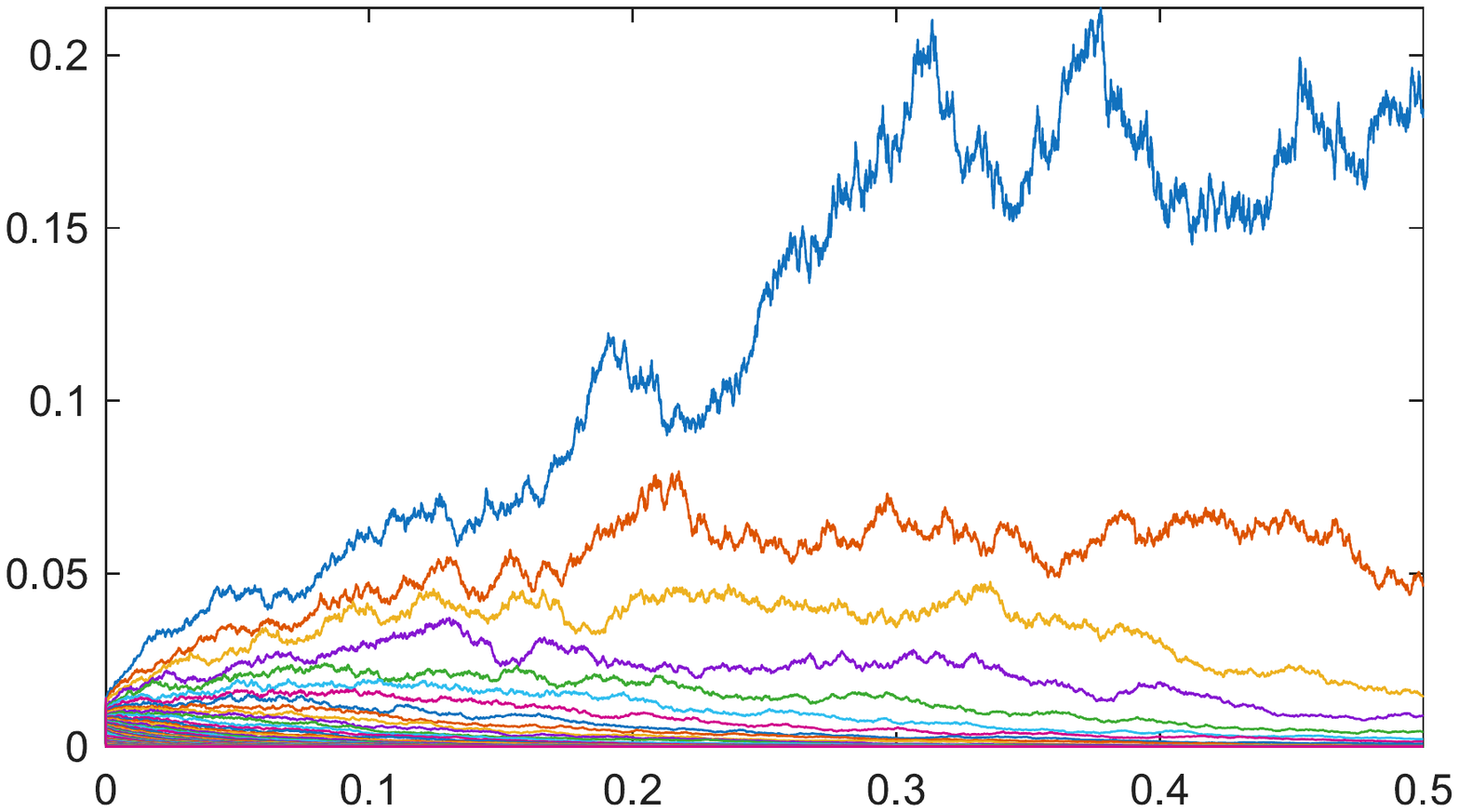}
 \end{minipage}
 \hspace{0.02\textwidth}
  \begin{minipage}[c]{0.48\textwidth}
\centering
\includegraphics[width=0.95\linewidth]{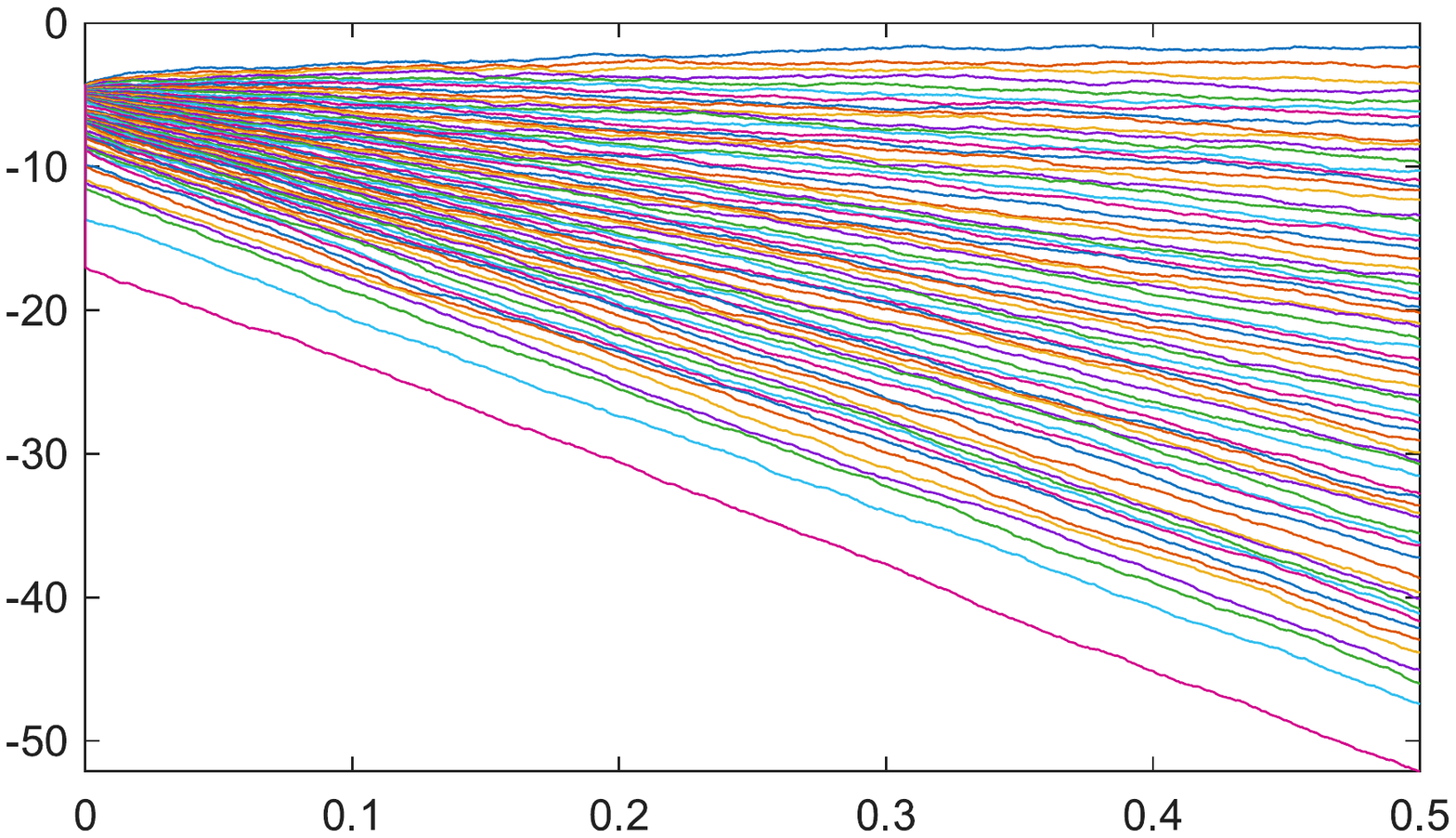}
 \end{minipage}
\caption[]{Simulation of the process  $N^{-1}\mathsf{x}^{(N)}$ with $\theta = -2$   and $N = 70$ particles started from a random initial configuration  (left), along with the log-transformed dynamics (right).} 
\label{rscSVofGL_NBm-X_0=rand}
\end{figure}

Our main results below establish, respectively, the fixed-time and multitime determinantal structure of
the limiting process from every initial condition  and provide explicit formulas for its
correlation kernels.

In what follows, we suppress the dependence of the correlation kernels on
the initial state for notational simplicity.

\begin{thm}\label{thmIntro-fixedTimeKernel}
    Let $\theta\in\mathbb{R}$ and
    $\upsilon
=
\left(
\left(
x_i
\right)_{i\in\mathbb{N}},
\gamma
\right)\in\Upsilon$.
    Consider the process $\mathsf{X}(\sbullet)
=
\left(
\left(
\mathsf{x}_i(\sbullet)
\right)_{i\in\mathbb{N}},
\bm{\gamma}(\sbullet)
\right)$ from
    Theorem~\ref{thm-pathSpaceConv}, started from $\upsilon$, and let $\mathfrak{x}_i(\sbullet)\defeq\log\mathsf{x}_i(\sbullet)$. Then,
    for every $t>0$, the point configuration
    \begin{align*}\Xi_t
    \defeq\sum_{i:\,\mathfrak{x}_i(t)>-\infty}
    \delta_{\mathfrak{x}_i(t)} 
    \end{align*}
    is a
    determinantal point process on $\mathbb{R}$ with correlation kernel
    $\mathcal{K}_t^{(\theta)}$, given, for any $u,v\in\mathbb{R}$, by
    \begin{align}\label{limitingFixedTimeKernel-Intro}
\mathcal{K}_t^{(\theta)}(u,v)
=
\frac{1}{(2\pi \mathrm{i})^2t}
\int_{\Gamma}dw
\int_{\Lambda}dz
\frac{
\exp\left(
\frac{w^2-z^2}{2t}
+\frac{z(u-\theta t/2)-w(v-\theta t/2)}{t}
-\frac{z-w}{2}
\right)
}{
\mathrm{e}^{w-z}-1
}
\frac{\varphi_{\upsilon}(\mathrm{e}^{-w})}
{\varphi_{\upsilon}(\mathrm{e}^{-z})}.
\end{align}
Here $\Gamma=L+\mathrm{i}\mathbb{R}$ is oriented upward, and
$\Lambda$ is the positively oriented half-strip contour
\begin{align*}
\Lambda
=
\{A+\mathrm{i}\vartheta:-\pi\le\vartheta\le\pi\}
\cup
\{x\pm\mathrm{i}\pi:x\le A\},
\end{align*}
where $\log x_1<A<L$, with orientation inherited from positively
oriented rectangular truncations.
\end{thm}

Observe that zero initial coordinates introduce no additional finite poles in the integrand.
In particular, if $x_1=0$, then all particle coordinates vanish and,
by the convention $\log0=-\infty$, the  condition
$\log x_1<A<L$ reduces simply to $A<L$.
Note also that the point measure $\Xi_t$ is indeed locally finite. In fact, we have
$\sum_{i\in\mathbb N}\mathsf{x}_i(t)\le\bm{\gamma}(t)<\infty$,
and hence only finitely many logarithmic particles can lie in any
compact subset of $\mathbb{R}$.

\begin{thm}\label{thmIntro-multitimeKernel}
 Let $\theta\in\mathbb{R}$ and
  $\upsilon
=
\left(
\left(
x_i
\right)_{i\in\mathbb{N}},
\gamma
\right)\in\Upsilon$.
    Consider the process $\mathsf{X}(\sbullet)
=
\left(
\left(
\mathsf{x}_i(\sbullet)
\right)_{i\in\mathbb{N}},
\bm{\gamma}(\sbullet)
\right)$ from
    Theorem~\ref{thm-pathSpaceConv}, started from $\upsilon$, and let $\mathfrak{x}_i(\sbullet)\defeq\log\mathsf{x}_i(\sbullet)$. 
    Then, for
every $m\in\mathbb{N}$ and $0<t_1<\cdots<t_m$, the space-time point
configuration
\begin{align*}
\Xi_{t_1,\ldots,t_m}
\defeq
\sum_{k=1}^{m}
\sum_{i:\,\mathfrak{x}_i(t_k)>-\infty}
\delta_{(t_k,\mathfrak{x}_i(t_k))}
\end{align*}
is a determinantal point process on
$\{t_1,\ldots,t_m\}\times\mathbb{R}$. Its extended correlation kernel is given, 
for $s,t\in\{t_1,\ldots,t_m\}$ and $u,v\in\mathbb{R}$, by
\begin{align}
\label{limitingMultitimeKernel-Intro}
\mathcal{K}^{(\theta)}(s,u;t,v)
={}&
-\mathbf{1}_{\{s<t\}}\mathcal{P}_{s,t}(u,v)
\nonumber\\
&+
\frac{1}{(2\pi\mathrm{i})^2\sqrt{st}}
\int_{\Gamma}dw
\int_{\Lambda}dz\,
\frac{1}{\mathrm{e}^{w-z}-1}
\nonumber\\
&\;\;\times
\exp\left(
\frac{w^2}{2s}
-
\frac{w(u-\theta s/2)}{s}
-
\frac{z^2}{2t}
+
\frac{z(v-\theta t/2)}{t}
-
\frac{z-w}{2}
\right)
\frac{
\varphi_{\upsilon}(\mathrm{e}^{-w})
}{
\varphi_{\upsilon}(\mathrm{e}^{-z})
},
\end{align}
where
\begin{align*}
\mathcal{P}_{s,t}(u,v)
\defeq
\frac{1}{\sqrt{2\pi(t-s)}}
\exp\left(
-\frac{(tu-sv)^2}{2st(t-s)}
\right),
\qquad
s<t.
\end{align*}
Here $\Gamma=L+\mathrm{i}\mathbb{R}$ and $\Lambda$ are the same
oriented contours as in Theorem~\ref{thmIntro-fixedTimeKernel}, with
$\log x_1<A<L$.
\end{thm}

Note that the determinantal structure in
Theorems~\ref{thmIntro-fixedTimeKernel} and
\ref{thmIntro-multitimeKernel} holds for every deterministic initial
state in $\Upsilon$, the full state space of the limiting enhanced
Feller process $\mathsf{X}$. The associated Laguerre--P\'olya function $\varphi_\upsilon$, arising
as the limiting reverse characteristic polynomial, appears explicitly
in the kernels and records, through the parameter $\gamma$,
information not determined by the particle configuration. In particular, states with
identical particle coordinates but different values of $\gamma$ yield
different kernel formulas. This distinction is also manifest
dynamically: such states give rise to particle processes with
different laws, despite satisfying the same ISDE
\cite{AssiotisMirsajjadi2026}.

 Finally, for the identity initial condition
$\upsilon_{\mathrm{Id}}=\bigl((0)_{i\in\mathbb{N}},1\bigr)$
and zero drift, 
 our formulas provide alternative contour-integral representations for correlation 
kernels of the limiting dynamics studied in \cite{Ahn}; Corollary~\ref{cor-identityInitialKernel}
specializes these formulas, and
Proposition~\ref{prop-identity-Ahn} identifies the resulting
extended kernel with Ahn's kernel up to an explicit Gaussian
conjugation. The analysis in \cite{Ahn} exploits the special
structure of this initial condition and does not immediately
extend to general initial states.

We next present further consequences of the preceding results. We
begin by introducing  notation that accounts for degenerate states in 
 $\Upsilon$. For such states, some exponential coordinates
may remain identically zero, or equivalently, the corresponding
logarithmic coordinates may remain equal to $-\infty$. 

For $\upsilon=(\bm{x},\gamma)\in\Upsilon$, write
\begin{align*}
\mathfrak{d}(\upsilon)
\defeq
\gamma-\sum_{i=1}^{\infty}x_i,
\qquad
\mathfrak{r}(\upsilon)
\defeq
\#\left\{
i\in\mathbb{N}:x_i>0
\right\},
\end{align*}
and define
\begin{align}\label{activeIndexSet}
\mathcal{I}(\upsilon)
\defeq
\begin{cases}
\llbracket\mathfrak{r}(\upsilon)\rrbracket,
&
\mathfrak{d}(\upsilon)=0
\textnormal{ and }
\mathfrak{r}(\upsilon)<\infty,
\\[1mm]
\mathbb{N},
&
\mathfrak{d}(\upsilon)>0
\textnormal{ or }
\mathfrak{r}(\upsilon)=\infty,
\end{cases}
\end{align}
with the convention that $\{1,\ldots,0\}=\varnothing$. 
We call the coordinates indexed by $\mathcal{I}(\upsilon)$ active. 
As will be shown in Proposition~\ref{prop-positiveCoordinates}, these
are precisely the indices of the coordinates of the  limiting dynamics that are strictly positive at every fixed positive time  and hence give rise to finite-valued logarithmic curves
on $(0,\infty)$.

Throughout, we denote by $\mathbb{P}_{\upsilon}$ the law of the
process $\mathsf{X}(\sbullet)$ from Theorem~\ref{thm-pathSpaceConv}, started at 
$\upsilon\in\Upsilon$, and by 
$\mathbb{E}_{\upsilon}$ the expectation with respect to this law.

We now turn to the Gibbs resampling property of the limiting dynamics. 
We use the standard terminology for Gibbsian line ensembles
\cite{CorwinHammond}.  
The following result establishes the Gibbs resampling property for
arbitrary initial states, extending the corresponding result of
\cite{AssiotisMirsajjadi2026} to initial configurations with possibly
coinciding coordinates.
\begin{cor}\label{corIntro-GibbsProperty}
Let $\theta\in\mathbb{R}$ and
$\upsilon\in\Upsilon$. Consider the process
$\mathsf{X}(\sbullet)$ from Theorem~\ref{thm-pathSpaceConv}, started from
$\upsilon$, and let  $\mathfrak{x}_i(\sbullet)\defeq\log\mathsf{x}_i(\sbullet)$. 
Then the line ensemble
$(\mathfrak{x}_i(t))_{i\in\mathcal{I}(\upsilon),\,t>0}$
satisfies the Brownian Gibbs property on
$\mathcal{I}(\upsilon)\times(0,\infty)$; that is, the Gibbs resampling
property, with respect to the law of the Brownian bridge, holds on every compact time interval contained in
$(0,\infty)$.
In particular, we have
\begin{align}\label{nonCollision}
\mathbb{P}_{\upsilon}
\left(
\mathfrak{x}_i(t)>\mathfrak{x}_{i+1}(t)
\textnormal{ for every }
i,i+1\in\mathcal{I}(\upsilon)
\textnormal{ and every }t>0
\right)
=
1.
\end{align}
\end{cor}

Consequently, even when some active particle coordinates coincide initially,
they separate instantaneously and are strictly ordered at every positive time.

As a further consequence, we obtain an SDE description of the active
coordinates from arbitrary initial configurations. 
\begin{thm}\label{thmIntro-ISDE}
Let $\theta\in\mathbb{R}$ and
$\upsilon=(\bm{x},\gamma)\in\Upsilon$. Consider the process
$\mathsf{X}(\sbullet)$ from Theorem~\ref{thm-pathSpaceConv}, started
from $\upsilon$. Then there exist independent standard Brownian motions
$(\mathsf{w}_i)_{i\in\mathcal{I}(\upsilon)}$ such that, almost surely,
for every $i\in\mathcal{I}(\upsilon)$ and $t\ge0$,
\begin{align}
\mathsf{x}_i(t)
=
x_i
+
\int_0^t\mathsf{x}_i(s)\,\mathrm{d}\mathsf{w}_i(s)
+
\frac{\theta}{2}\int_0^t\mathsf{x}_i(s)\,\mathrm{d}s
+
\int_0^t\sum_{j\in\mathcal{I}(\upsilon)\setminus\{i\}}
\frac{\mathsf{x}_i(s)\mathsf{x}_j(s)}
{\mathsf{x}_i(s)-\mathsf{x}_j(s)}
\,\mathrm{d}s,
\label{ISDE-exponential}
\end{align}
where the singular interaction integral is understood as the following almost surely improper limit
\begin{align*}
\int_0^t\sum_{j\in\mathcal{I}(\upsilon)\setminus\{i\}}
\frac{\mathsf{x}_i(s)\mathsf{x}_j(s)}
{\mathsf{x}_i(s)-\mathsf{x}_j(s)}
\,\mathrm{d}s
\defeq
\lim_{\epsilon\downarrow0}
\int_\epsilon^t\sum_{j\in\mathcal{I}(\upsilon)\setminus\{i\}}
\frac{\mathsf{x}_i(s)\mathsf{x}_j(s)}
{\mathsf{x}_i(s)-\mathsf{x}_j(s)}
\,\mathrm{d}s.
\end{align*}
\end{thm}

The corresponding logarithmic formulation follows from It\^o's formula.

\begin{cor}\label{corIntro-logISDE}
Under the assumptions of Theorem~\ref{thmIntro-ISDE}, almost surely,
for every $i\in\mathcal{I}(\upsilon)$, $t>0$, and $0<\epsilon<t$, we have
\begin{align}
\mathfrak{x}_i(t)
=
\mathfrak{x}_i(\epsilon)
+
\mathsf{w}_i(t)-\mathsf{w}_i(\epsilon)
+
\frac{\theta-1}{2}(t-\epsilon)
+
\int_\epsilon^t\sum_{j\in\mathcal{I}(\upsilon)\setminus\{i\}}
\frac{1}
{\mathrm{e}^{\mathfrak{x}_i(s)-\mathfrak{x}_j(s)}-1}
\,\mathrm{d}s.
\label{logISDE-positiveTime}
\end{align}
If 
$x_i>0$
for every $i\in\mathcal{I}(\upsilon)$, then
\eqref{logISDE-positiveTime} extends to $\epsilon=0$:
\begin{align}
\mathfrak{x}_i(t)
=
\log x_i
+
\mathsf{w}_i(t)
+
\frac{\theta-1}{2}t
+
\int_0^t\sum_{j\in\mathcal{I}(\upsilon)\setminus\{i\}}
\frac{1}
{\mathrm{e}^{\mathfrak{x}_i(s)-\mathfrak{x}_j(s)}-1}
\,\mathrm{d}s,
\label{logISDE-timeZero}
\end{align}
where the singular interaction integral is understood  as the following almost surely improper limit 
\begin{align*}
\int_0^t\sum_{j\in\mathcal{I}(\upsilon)\setminus\{i\}}
\frac{1}
{\mathrm{e}^{\mathfrak{x}_i(s)-\mathfrak{x}_j(s)}-1}
\,\mathrm{d}s
\defeq
\lim_{\epsilon\downarrow0}
\int_\epsilon^t\sum_{j\in\mathcal{I}(\upsilon)\setminus\{i\}}
\frac{1}
{\mathrm{e}^{\mathfrak{x}_i(s)-\mathfrak{x}_j(s)}-1}
\,\mathrm{d}s.
\end{align*}
\end{cor}

When some active initial coordinates vanish, the logarithmic process may
be viewed as entering from $-\infty$: all active logarithmic coordinates
are finite at every positive time, and the logarithmic equation holds on
every interval bounded away from zero.

\subsection{Outline of the proofs}

The derivation of the limiting correlation kernels proceeds by first
obtaining suitable double-contour formulas for the fixed-time and multitime
correlation kernels of the finite-dimensional processes and then
analyzing these formulas under the edge scaling leading to the limiting
dynamics.

It is worth highlighting that, for the identity initial condition
studied in \cite{Ahn}, the correlation kernel of the logarithms of the
squared singular values can be related, through a reciprocal
relation, to that of Dyson Brownian motion started from equidistant
initial conditions; see
\cite{Jones-OConnell,Katori-kernel,Takahashi-Katori}. The latter has a particularly tractable
form and is well suited to asymptotic analysis.
   For general initial states, no analogous reduction is apparent. A 
different structure nevertheless emerges: the initial condition is incorporated through the associated Laguerre--P\'olya
function $\varphi_{\upsilon}$.

For the fixed-time kernel, we begin with non-colliding Brownian bridges
having centered arithmetic terminal points. Adapting the residue
calculation in the proof of
\cite[Theorem~2.2]{Johansson2004} to these terminal points yields a
double-contour representation for a correlation kernel of the bridge
ensemble. Choosing the terminal configuration to scale with the
terminal time and subsequently taking the infinite-time limit gives the
desired finite-$N$ kernel.

Passing from this formula to the infinite-particle kernel requires
care: both the poles and the natural integration contour depend on
$N$. Moreover, after the edge scaling, the finite-$N$ formula contains
exponential factors of the form $\mathrm{e}^{N(w-z)}$ that do not
have limits when considered separately. Combining these factors and
conjugating the kernel by a factor that leaves its correlation
functions unchanged removes the apparent divergence. The dependence
on the finite-$N$ initial condition then takes the form
\begin{equation*}
\frac{\varphi_{\upsilon^{(N)}}(\mathrm{e}^{-w})}
     {\varphi_{\upsilon^{(N)}}(\mathrm{e}^{-z})}.
\end{equation*}
The locally uniform convergence of the entire functions cannot yet be
used directly because the contours depend on $N$.  
We next deform the $N$-dependent contour enclosing the finite-$N$
poles to a fixed half-strip contour $\Lambda$. The left vertical side
can be sent to $-\infty$, since its contribution vanishes. Thus all
finite-$N$ kernels are represented on the same pair of contours.
The convergence $\upsilon^{(N)}\to\upsilon$ in $\Upsilon$ implies
$\varphi_{\upsilon^{(N)}}\longrightarrow\varphi_\upsilon$ locally uniformly, and
hence pointwise convergence of the integrands on these contours,
where the limiting denominator does not vanish. Because $\Lambda$
is unbounded, pointwise convergence alone does not suffice. We obtain
an integrable majorant uniform in $N$: Gaussian decay controls the
vertical $w$-contour and the unbounded horizontal parts of $\Lambda$,
while the product structure of $\varphi_{\upsilon^{(N)}}$, together
with the uniform bound on $N^{-1}\sum_j x_j^{(N)}$, controls the ratio
of entire functions. Dominated convergence then yields locally
uniform convergence of the finite-$N$ kernels to the kernel in
Theorem~\ref{thmIntro-fixedTimeKernel}.

The multitime proof follows the same limiting argument. Starting from
the Eynard--Mehta formula, we derive a double-contour representation
for the extended kernel of the non-colliding bridge ensemble and take
the infinite-terminal-time limit. After edge scaling, the same
cancellation and contour deformation reduce the $N$-dependence to
$\varphi_{\upsilon^{(N)}}$ on fixed contours. The preceding uniform
bounds then justify passage to the limit in the double integral; the
time-ordering term converges separately. This yields the extended
kernel in Theorem~\ref{thmIntro-multitimeKernel}. Finally, locally
uniform convergence of the kernels, together with the convergence of
the dynamics in Theorem~\ref{thm-pathSpaceConv}, establishes the
determinantal structure of the limiting fixed-time and multitime point
processes.

A notable feature of the resulting formulas is the direct appearance of
the Laguerre--P\'olya entire function $\varphi_\upsilon$, encoding the 
dependence of the kernels on the initial state. 
For the identity initial condition and zero drift,
a geometric-series expansion and the Hankel formula
for the reciprocal Gamma function identify our
extended kernel with Ahn's kernel up to Gaussian
conjugation.

We finally turn to the Gibbs resampling and ISDE results. We first
establish positivity of the active coordinates and use the
determinantal structure to prove their strict ordering at each fixed
positive time. These properties, together with the Markov property and
the results of \cite{AssiotisMirsajjadi2026}, allow us to extend the
Gibbs resampling property and the associated ISDE description to
arbitrary initial data. In particular, this provides the ISDE
formulation for configurations with coinciding particle coordinates;
the singular interaction at time zero is understood as an improper
time integral.

The explicit kernels may further provide a route to obtaining fixed-time
rigidity estimates which, together with the path regularity provided by
the Gibbs resampling property, could then imply that the law of the
dynamics is concentrated on a suitable rigid-path space in which uniqueness
may be investigated.

It would be of interest to establish analogous determinantal
structures, including explicit correlation kernels, and suitable
Gibbs resampling properties for the other two limiting dynamics
studied in \cite{AssiotisMirsajjadi2026}. We leave these questions
for future work.

\paragraph*{Acknowledgements}
The author thanks Theodoros Assiotis for past discussions
concerning the problem and for comments on an initial draft of the
manuscript. The author also thanks Martin Dindo\v{s} for his feedback.
Support from the EPSRC through a PhD studentship during the author's
doctoral studies is acknowledged.

\section{Correlation kernel}

\subsection{Fixed-time correlation kernel}

We begin by deriving the fixed-time correlation kernel for the
corresponding finite-dimensional dynamics. We first identify the
fixed-time distribution of the non-colliding Brownian motions with drift
as the infinite-horizon limit of the corresponding
distribution of non-colliding Brownian bridges, and then pass to the
limit in the associated correlation kernels.

Let us denote by $(p_t)_{t\ge0}$ the heat kernel:
\begin{align*}
p_t(x,y)
=
\frac{1}{\sqrt{2\pi t}}
\exp\left(
-\frac{(x-y)^2}{2t}
\right),
\quad 
x,y\in\mathbb{R},
\end{align*}
namely, the standard Brownian transition density function with respect
to Lebesgue measure.

\begin{lem}\label{lem-driftedBmKernel}
Let $N\in\mathbb{N}$ and take
$\bm{a}=(a_i)_{i=1}^N\in\mathbb{W}_N^\circ$.
Consider non-colliding drifted Brownian motions started from
$\bm{a}$ with drift vector
$\bm{\mu}=(\mu_i)_{i=1}^N\in\mathbb{W}_N^\circ$, where
\begin{align}\label{centeredDrifts}
\mu_i
=
\frac{N+1}{2}-i,
\quad 
i\in\Ibr{N}.
\end{align}
Their positions at time $t>0$ form a determinantal point process with
correlation kernel, for $u,v\in\mathbb R$, given by
\begin{align}\label{driftedBmKernel}
\mathcal{K}_t^N(u,v)
={}&
\frac{1}{(2\pi\mathrm{i})^2t}
\int_{\Gamma}dw
\int_{\Lambda}dz\,
\frac{1}{\mathrm{e}^{w-z}-1}
\exp\left(
\frac{(w-v)^2-(z-u)^2}{2t}
+
\frac{N-1}{2}(z-w)
\right)
\nonumber\\
&\quad \times
\prod_{j=1}^{N}
\frac{\mathrm{e}^w-\mathrm{e}^{a_j}}
{\mathrm{e}^z-\mathrm{e}^{a_j}}.
\end{align}
Here $\Gamma=L+\mathrm{i}\mathbb{R}$ is oriented upward, and
$\Lambda$ is a positively oriented simple closed contour enclosing
$a_1,\ldots,a_N$, but no other poles of the $z$-integrand, and lying
strictly to the left of $\Gamma$.
\end{lem}

\begin{proof}
We first show that, at every fixed time $t>0$, the distribution of the
process of interest arises as the limit of the corresponding
distribution of non-colliding Brownian bridges. Let $T>t$, and consider
$N$ Brownian bridges on the time interval $[0,T]$, starting from
$\bm{a}=(a_i)_{i=1}^N\in\mathbb{W}_N^\circ$, ending at
$\bm{b}^{(T)}
=\big(b_i^{(T)}\big)_{i=1}^N\in\mathbb{W}_N^\circ$,
and conditioned not to intersect.
By the Karlin--McGregor formula \cite{KarlinMcGregor}, their positions
$\bm{x}=(x_i)_{i=1}^N$ at time $t$ have density with respect
to Lebesgue measure on $\mathbb{W}_N^\circ$ given by
\begin{align*}
\frac{
\det\left(p_t(a_i,x_j)\right)_{i,j=1}^{N}
\det\left(p_{T-t}(x_i,b_j^{(T)})\right)_{i,j=1}^{N}
}{
\det\left(p_T(a_i,b_j^{(T)})\right)_{i,j=1}^{N}
}.
\end{align*}
This is a biorthogonal ensemble and therefore a determinantal point
process.

Now, let us take the endpoints as
$\bm{b}^{(T)}=T\bm{\mu}$.
The transition density from $\bm{a}$ at time $0$ to
$\bm{x}$ at time $t$ is then equal to
\begin{align*}
q_t^{(T)}(\bm{a},\bm{x})
\defeq
\det\left(p_t(a_i,x_j)\right)_{i,j=1}^{N}
\frac{
\det\left(p_{T-t}(x_i,T\mu_j)\right)_{i,j=1}^{N}
}{
\det\left(p_T(a_i,T\mu_j)\right)_{i,j=1}^{N}
}.
\end{align*}
Using the Gaussian expression for the heat kernel, we have
\begin{align*}
\det\left(
p_{T-t}(x_i,T\mu_j)
\right)_{i,j=1}^{N}
={}&
(2\pi(T-t))^{-N/2}
\exp\left(
-\frac{1}{2(T-t)}
\sum_{i=1}^{N}x_i^2
\right)
\\
&\quad \times
\exp\left(
-\frac{T^2}{2(T-t)}
\sum_{j=1}^{N}\mu_j^2
\right)
\det\left(
\mathrm{e}^{\frac{T}{T-t}\mu_jx_i}
\right)_{i,j=1}^{N},
\end{align*}
and
\begin{align*}
\det\left(
p_T(a_i,T\mu_j)
\right)_{i,j=1}^{N}
={}&
(2\pi T)^{-N/2}
\exp\left(
-\frac{1}{2T}
\sum_{i=1}^{N}a_i^2
\right)
\exp\left(
-\frac{T}{2}
\sum_{j=1}^{N}\mu_j^2
\right)
\det\left(
\mathrm{e}^{\mu_ja_i}
\right)_{i,j=1}^{N}.
\end{align*}
Consequently, we have, as $T\to\infty$,
\begin{align*}
\frac{
\det\left(p_{T-t}(x_i,T\mu_j)\right)_{i,j=1}^{N}
}{
\det\left(p_T(a_i,T\mu_j)\right)_{i,j=1}^{N}
}
\longrightarrow
\mathrm{e}^{-\frac{t}{2}|\bm{\mu}|^2}
\frac{
h_{\bm{\mu}}(\bm{x})
}{
h_{\bm{\mu}}(\bm{a})
},
\end{align*}
where
\begin{align*}
h_{\bm{\mu}}(\bm{x})
\defeq
\det\left(
\mathrm{e}^{\mu_jx_i}
\right)_{i,j=1}^{N}.
\end{align*}
It follows that
\begin{align}\label{driftedBmTransition}
\lim_{T\to\infty}
q_t^{(T)}(\bm{a},\bm{x})
=
\mathrm{e}^{-\frac{t}{2}|\bm{\mu}|^2}
\frac{
h_{\bm{\mu}}(\bm{x})
}{
h_{\bm{\mu}}(\bm{a})
}
\det\left(
p_t(a_i,x_j)
\right)_{i,j=1}^{N}.
\end{align}
The right-hand side of \eqref{driftedBmTransition} is exactly the
transition density of non-colliding Brownian motions with drift vector
$\bm{\mu}$, started from $\bm{a}$
\cite{BBO,Jones-OConnell}.
Since the right-hand side of \eqref{driftedBmTransition} is a
probability density, Scheff\'e's lemma
\cite[Section~5.10]{Williams1991} implies convergence in $L^1$ of the
fixed-time joint densities.

We now obtain their correlation kernel. More generally, consider the
terminal points
\begin{align}\label{multitimeEndpoints}
b_i
=
\beta\mu_i
=
\beta\left(
\frac{N+1}{2}-i
\right),
\quad 
i\in\Ibr{N},
\quad 
\textnormal{ with }\beta>0.
\end{align}
Following the residue calculation in the proof of
\cite[Theorem~2.2]{Johansson2004}, applied to the centered endpoints
above (although this theorem is formulated for an odd
number of particles, the residue calculation underlying the contour
representation is valid for arbitrary $N\in\mathbb{N}$), a correlation kernel is given by
\begin{align}\label{JohKernel}
\mathcal{K}_t^{N,T}(u,v)
={}&
\frac{
\beta\mathrm{e}^{(u^2-v^2)/(2(T-t))}
}{
(2\pi\mathrm{i})^2tT
}
\int_{\Gamma}dw
\int_{\Lambda}dz\,
\frac{1}{
\mathrm{e}^{\beta(w-z)/T}-1
}
\nonumber\\
&\quad \times
\exp\left(
\frac{T-t}{2tT}
\left[
\left(
w-\frac{T}{T-t}v
\right)^2
-
\left(
z-\frac{T}{T-t}u
\right)^2
\right]
+
\frac{\beta(N-1)}{2T}(z-w)
\right)
\nonumber\\
&\quad \times
\prod_{j=1}^{N}
\frac{
\mathrm{e}^{\beta w/T}
-
\mathrm{e}^{\beta a_j/T}
}{
\mathrm{e}^{\beta z/T}
-
\mathrm{e}^{\beta a_j/T}
}.
\end{align}
Here $\Gamma=L+\mathrm{i}\mathbb{R}$ is oriented upward, and
$\Lambda$ is a positively oriented simple closed contour surrounding
$a_1,\ldots,a_N$, but no other poles of the $z$-dependent product, and
lying strictly to the left of $\Gamma$.

To obtain the desired kernel, we set $\beta=T$ and choose an admissible
contour $\Lambda$, independent of $T$, contained in
$\{z\in\mathbb{C}:|\Im{z}|<\pi\}$.
We then let $T\to\infty$. Observe that the integrand converges pointwise
to the corresponding limiting integrand. Moreover, writing
$w=L+\mathrm{i}\tau$, for all sufficiently large $T$,
\begin{align*}
\left|
\exp\left(
\frac{T-t}{2tT}
\left(
w-\frac{T}{T-t}v
\right)^2
\right)
\right|
\le
C\mathrm{e}^{-c_0\tau^2}
\end{align*}
for some constants $C,c_0>0$ independent of $T$. Since $\Lambda$ is compact,
is contained in $|\Im{z}|<\pi$, surrounds but does not intersect
the points $a_1,\ldots,a_N$, and lies strictly to the left of
$\Gamma$, all the remaining factors are uniformly bounded on
$\Gamma\times\Lambda$ for all sufficiently large $T$. Hence, the
integrand is dominated by an integrable function independent of $T$.

The preceding estimates are uniform when $(u,v)$ ranges over a compact
subset of $\mathbb{R}^2$. The dominated convergence theorem then implies
\begin{align*}
\mathcal{K}_t^{N,T}(u,v)
\longrightarrow
\mathcal{K}_t^N(u,v),
\quad 
T\to\infty,
\end{align*}
locally uniformly in $(u,v)$, where $\mathcal{K}_t^N$ is the kernel in
\eqref{driftedBmKernel}. Consequently, the determinantal correlation
functions of the bridge ensembles converge locally uniformly to
\begin{align}\label{det(K)}
\det\left(
\mathcal{K}_t^N(u_i,u_j)
\right)_{i,j=1}^{m},
\quad 
m\in\Ibr{N}.
\end{align}

On the other hand, the $L^1$ convergence of the joint densities
established above implies convergence in $L^1$ of the corresponding
correlation functions to those of the non-colliding Brownian motions
with drift. The two limits must therefore agree almost everywhere.
Thus, for every $m\in\Ibr{N}$, the correlation function of the latter
process is given by the determinant in \eqref{det(K)}, and
$\mathcal{K}_t^N$ is a correlation kernel for it. This completes the
proof.
\end{proof}

\subsubsection{Proof of Theorem~\ref{thmIntro-fixedTimeKernel}}

We now establish the fixed-time determinantal structure and limiting kernel stated in Theorem~\ref{thmIntro-fixedTimeKernel}.
\begin{proof}[Proof of Theorem~\ref{thmIntro-fixedTimeKernel}]
Observe first from \eqref{logSingularValuesSDE} that the parameter
$\theta$ translates every logarithmic particle by the same
deterministic amount $\theta t/2$. Hence,
\begin{align*}
\mathcal{K}_t^{(\theta)}(u,v)
=
\mathcal{K}_t^{(0)}
\left(
u-\frac{\theta t}{2},
v-\frac{\theta t}{2}
\right).
\end{align*}
It is therefore enough to perform the computation for $\theta=0$,
which we do below; the general case follows from this deterministic
translation. In the remainder of this proof, we write
$\mathcal{K}_t=\mathcal{K}_t^{(0)}$.

Choose initial configurations
$\bm{x}^{(N)}=\big(x_i^{(N)}\big)_{i=1}^N\in
\mathbb{W}_{N,+}^\circ$
satisfying \eqref{IC-conv-exp}, and set
\begin{align*}
\alpha_i^{(N)}\defeq\log x_i^{(N)}, \quad  i\in\Ibr{N}.
\end{align*}

Fix $L>A>\log x_1$.
From \eqref{IC-conv-exp} we have
$N^{-1}x_1^{(N)}
\xrightarrow{N\to\infty}
x_1$. 
Thus, there exists $N_0\in\mathbb{N}$ such that
\begin{align*}
\alpha_1^{(N)}-\log N<A,
\qquad 
N\ge N_0.
\end{align*}
For $\theta=0$, the process $\mathfrak{y}^{(N)}$ is equal in
distribution to the non-colliding Brownian motion with drift vector
$\boldsymbol{\mu}$ given in \eqref{centeredDrifts}; see
\cite{Jones-OConnell,BBO}. Hence, by
Lemma~\ref{lem-driftedBmKernel}, for every $N\ge N_0$, a correlation
kernel of $\mathfrak{y}^{(N)}$ started from
$\boldsymbol{\alpha}^{(N)}
=\big(\alpha_i^{(N)}\big)_{i=1}^N$ 
can be written as, for $u,v\in\mathbb{R}$,
\begin{align}\label{finiteN-Kernel}
\mathcal{K}_t^N(u,v)
={}&
\frac{1}{(2\pi\mathrm{i})^2t}
\int_{\Gamma_N}dw
\int_{\Lambda_N}dz\,
\frac{1}{\mathrm{e}^{w-z}-1}
\exp\left(
\frac{(w-v)^2-(z-u)^2}{2t}
+
\frac{N-1}{2}(z-w)
\right)
\nonumber\\
&\quad \times
\prod_{j=1}^{N}
\frac{
\mathrm{e}^w-\mathrm{e}^{\alpha_j^{(N)}}
}{
\mathrm{e}^z-\mathrm{e}^{\alpha_j^{(N)}}
},
\end{align}
where
\begin{align}\label{Gamma_N}
\Gamma_N
=
\Gamma+\log N,
\qquad
\Gamma
=
L+\mathrm{i}\mathbb{R},
\end{align}
and the positively oriented contour $\Lambda_N$ is given by
\begin{align}\label{Lambda_N}
\Lambda_N
=
\partial
\big\{
z:
\log N-R_N\le\Re{(z)}\le A+\log N,
\ 
-\pi\le\Im{(z)}\le\pi
\big\},
\end{align}
with $R_N>0$ chosen so that
\begin{align*}
\log N-R_N
<
\alpha_N^{(N)}.
\end{align*}
Observe in particular that $\Gamma_N$ is an admissible choice of the
vertical contour in Lemma~\ref{lem-driftedBmKernel} and that, by contour
deformation, the $z$-contour therein can be replaced by $\Lambda_N$.
Indeed, for every $N\ge N_0$, the only poles of the $z$-integrand
enclosed by $\Lambda_N$ are
\begin{align*}
z
=
\alpha_j^{(N)},
\quad
j\in\Ibr{N},
\end{align*}
while $\Gamma_N$ lies strictly to the right of $\Lambda_N$. 
In the remainder of the proof, we restrict to $N\ge N_0$.

We conjugate the kernel in \eqref{finiteN-Kernel} as follows:
\begin{align*}
\widetilde{\mathcal{K}}_t^N(u,v)
\defeq
\mathrm{e}^{u^2/(2t)}
\mathcal{K}_t^N(u,v)
\mathrm{e}^{-v^2/(2t)}.
\end{align*}
Observe that this preserves all determinantal correlation functions.
Expanding the squares then yields
\begin{align*}
\widetilde{\mathcal{K}}_t^N(u,v)
={}&
\frac{1}{(2\pi\mathrm{i})^2t}
\int_{\Gamma_N}dw
\int_{\Lambda_N}dz\,
\frac{1}{\mathrm{e}^{w-z}-1}
\exp\left(
\frac{w^2-z^2}{2t}
+
\frac{zu-wv}{t}
+
\frac{N-1}{2}(z-w)
\right)
\\
&\quad \times
\prod_{j=1}^{N}
\frac{
\mathrm{e}^w-\mathrm{e}^{\alpha_j^{(N)}}
}{
\mathrm{e}^z-\mathrm{e}^{\alpha_j^{(N)}}
}.
\end{align*}

We now apply the scaling under which the process converges. At the
level of the correlation kernel, this amounts to shifting the spatial
variables according to
\begin{align*}
u
\mapsto
u+\frac{Nt}{2}+\log N,
\qquad
v
\mapsto
v+\frac{Nt}{2}+\log N.
\end{align*}
Consequently, we get
\begin{align*}
\widetilde{\mathcal{K}}_t^N
&\left(
u+\frac{Nt}{2}+\log N,
v+\frac{Nt}{2}+\log N
\right)
\\
&
={}
\frac{1}{(2\pi\mathrm{i})^2t}
\int_{\Gamma_N}dw
\int_{\Lambda_N}dz\,
\frac{1}{\mathrm{e}^{w-z}-1}
\nonumber\\
&\,\times\exp\left(
\frac{w^2-z^2}{2t}
+
\frac{zu-wv}{t}
+
\left(
\frac N2+\frac{\log N}{t}
\right)(z-w)
+
\frac{N-1}{2}(z-w)
\right)
\prod_{j=1}^{N}
\frac{
\mathrm{e}^w-\mathrm{e}^{\alpha_j^{(N)}}
}{
\mathrm{e}^z-\mathrm{e}^{\alpha_j^{(N)}}
}.
\end{align*}
We then translate the contour variables by setting
\begin{align*}
z_{\mathrm{old}}
=
z+\log N,
\qquad
w_{\mathrm{old}}
=
w+\log N.
\end{align*}
Finally, multiplying the resulting kernel by the conjugation factor (note that this does not change
the correlation functions)
\begin{align*}
\mathrm{e}^{\frac{\log N}{t}(v-u)}
\end{align*}
and using
\begin{align*}
\prod_{j=1}^{N}
\frac{
\mathrm{e}^{w+\log N}
-
\mathrm{e}^{\alpha_j^{(N)}}
}{
\mathrm{e}^{z+\log N}
-
\mathrm{e}^{\alpha_j^{(N)}}
}
=
\mathrm{e}^{N(w-z)}
\prod_{j=1}^{N}
\frac{
1-N^{-1}x_j^{(N)}\mathrm{e}^{-w}
}{
1-N^{-1}x_j^{(N)}\mathrm{e}^{-z}
},
\end{align*}
we obtain the following rescaled finite-$N$ kernel:
\begin{align}\label{rscK_N}
\widehat{\mathcal{K}}_t^N(u,v)
=
\frac{1}{(2\pi\mathrm{i})^2t}
\int_{\Gamma}dw
\int_{\Lambda_N}dz\,
F_N(z,w),
\end{align}
where
\begin{align*}
F_N(z,w)
=
\frac{1}{\mathrm{e}^{w-z}-1}
\exp\left(
\frac{w^2-z^2}{2t}
+
\frac{zu-wv}{t}
-
\frac12(z-w)
\right)
\frac{
\varphi_{\upsilon^{(N)}}(\mathrm{e}^{-w})
}{
\varphi_{\upsilon^{(N)}}(\mathrm{e}^{-z})
},
\end{align*}
and 
\begin{align*}
\upsilon^{(N)}
\defeq
\left(
\left(
N^{-1}x_i^{(N)}
\right)_{i\in\mathbb{N}},
N^{-1}\sum_{i=1}^{N}x_i^{(N)}
\right)
\in\Upsilon
.
\end{align*}
In particular, we have
\begin{align*}
\varphi_{\upsilon^{(N)}}(\zeta)
=
\prod_{j=1}^{N}
\left(
1-N^{-1}x_j^{(N)}\zeta
\right), \quad \zeta\in\mathbb{C}.
\end{align*}
Also, with a slight abuse of notation, we again denote the translated
$z$-contour by $\Lambda_N$, given now by
\begin{align}\label{Lambda_N_new-contour}
\Lambda_N
=
\partial
\big\{
z:
-R_N\le\Re{(z)}\le A,
\ 
-\pi\le\Im{(z)}\le\pi
\big\}.
\end{align}
We note that the right side $\Re{(z)}=A$ of $\Lambda_N$ is independent
of $N$, and the poles of the $z$-dependent product enclosed by
$\Lambda_N$ are
\begin{align*}
z
=
\alpha_j^{(N)}-\log N,
\quad
j\in\Ibr{N}.
\end{align*}
The poles arising from
$(\mathrm{e}^{w-z}-1)^{-1}$ remain outside $\Lambda_N$, since they
have real part $\Re{(w)}=L>A$.

Now, for $R>0$ sufficiently large, define the positively oriented
rectangular contour
\begin{align*}
\Lambda(R)
\defeq
\partial
\big\{
z:
-R\le\Re{(z)}\le A,
\ 
-\pi\le\Im{(z)}\le\pi
\big\}.
\end{align*}
Note that, for each fixed $N$, moving the left side of the contour
$\Lambda_N$ appearing in \eqref{rscK_N} from $\Re{(z)}=-R_N$ to
$\Re{(z)}=-R$, with $R\ge R_N$, crosses no poles. Consequently, by contour
deformation, $\Lambda_N$ in \eqref{rscK_N} may be replaced by
$\Lambda(R)$ for every $R\ge R_N$.

We denote the left side of $\Lambda(R)$, equipped with its induced
orientation, by
\begin{align*}
\Lambda^{\mathrm{left}}(R)
\defeq
\{
-R-\mathrm{i}\vartheta:
-\pi\le\vartheta\le\pi
\}.
\end{align*}
We claim that, for every fixed $N$,
\begin{align*}
\int_{\Gamma}dw
\int_{\Lambda^{\mathrm{left}}(R)}dz\,
F_N(z,w)
\longrightarrow
0
\qquad
\textnormal{ as }\, R\to\infty.
\end{align*}
Indeed, since
$\Re(w-z)=L+R>0$, for all sufficiently large $R$,
\begin{align*}
\left|
\frac{1}{\mathrm{e}^{w-z}-1}
\right|
\le
\frac{1}{\mathrm{e}^{L+R}-1}
\le
C,
\qquad
(w,z)
\in
\Gamma\times\Lambda^{\mathrm{left}}(R),
\end{align*}
where the constant $C>0$ is independent of $R$ and $\tau$ and may change from line
to line.

Next, for $w=L+\mathrm{i}\tau$,
\begin{align}\label{GammaContourBound}
\Re\left(
\frac{w^2}{2t}
-
\frac{wv}{t}
+
\frac12w
\right)
=
\frac{L^2}{2t}
-
\frac{\tau^2}{2t}
-
\frac{Lv}{t}
+
\frac{L}{2}
\le
-\frac{\tau^2}{2t}
+
C.
\end{align}
Similarly, for
$z=-R-\mathrm{i}\vartheta$ with $|\vartheta|\le\pi$,
\begin{align*}
\Re\left(
-\frac{z^2}{2t}
+
\frac{zu}{t}
-
\frac12z
\right)
=
-\frac{R^2}{2t}
+
\frac{\vartheta^2}{2t}
-
\frac{Ru}{t}
+
\frac{R}{2}
\le
-\frac{R^2}{2t}
+
CR.
\end{align*}

Now consider
\begin{align*}
\Phi_N(z,w)
\defeq
\frac{
\varphi_{\upsilon^{(N)}}(\mathrm{e}^{-w})
}{
\varphi_{\upsilon^{(N)}}(\mathrm{e}^{-z})
}
=
\prod_{j=1}^{N}
\frac{
1-N^{-1}x_j^{(N)}\mathrm{e}^{-w}
}{
1-N^{-1}x_j^{(N)}\mathrm{e}^{-z}
}.
\end{align*}
On $\Gamma$, we have
\begin{align*}
\left|
1-N^{-1}x_j^{(N)}\mathrm{e}^{-w}
\right|
\le
1+N^{-1}x_j^{(N)}\mathrm{e}^{-L}, \quad j\in\Ibr{N}.
\end{align*}
Thus, for some $N$-dependent constant $c_N>0$ we have
\begin{align*}
\prod_{j=1}^{N}
\left|
1-N^{-1}x_j^{(N)}\mathrm{e}^{-w}
\right|
\le
c_N.
\end{align*}

Next, take $R$ large enough so that
\begin{align*}
N^{-1}x_N^{(N)}\mathrm{e}^{R}
\ge
2.
\end{align*}
Then, on $\Lambda^{\mathrm{left}}(R)$ we have
\begin{align*}
\left|
1-N^{-1}x_j^{(N)}\mathrm{e}^{-z}
\right|
&=
\left|
1-
N^{-1}x_j^{(N)}
\mathrm{e}^{R}
\mathrm{e}^{\mathrm{i}\vartheta}
\right|
\ge
N^{-1}x_j^{(N)}\mathrm{e}^{R}
-
1
\ge
\frac12
N^{-1}x_j^{(N)}\mathrm{e}^{R},
\quad
j\in\Ibr{N}.
\end{align*}
Therefore, for a possibly larger $c_N$, it holds that 
\begin{align*}
\prod_{j=1}^{N}
\frac{1}{
\left|
1-N^{-1}x_j^{(N)}\mathrm{e}^{-z}
\right|
}
\le
c_N\mathrm{e}^{-NR}.
\end{align*}
Hence,
\begin{align*}
|\Phi_N(z,w)|
\le
c_N\mathrm{e}^{-NR},
\qquad
(w,z)
\in
\Gamma\times\Lambda^{\mathrm{left}}(R).
\end{align*}
Combining the preceding estimates, we obtain
\begin{align*}
|F_N(z,w)|
\le
c_N
\mathrm{e}^{-\tau^2/(2t)}
\mathrm{e}^{-R^2/(2t)+CR}
\mathrm{e}^{-NR},
\qquad
(w,z)
\in
\Gamma\times\Lambda^{\mathrm{left}}(R).
\end{align*}
Since
$\operatorname{length}(\Lambda^{\mathrm{left}}(R))=2\pi$, we get
\begin{align*}
\int_{\Gamma}|dw|
\int_{\Lambda^{\mathrm{left}}(R)}|dz|\,
|F_N(z,w)|
\le
c_N
\mathrm{e}^{-R^2/(2t)+CR-NR}
\int_{\mathbb{R}}
\mathrm{e}^{-\tau^2/(2t)}
\,d\tau
\xrightarrow{R\to\infty}
0.
\end{align*}
Thus, the contribution of the left vertical side vanishes.

Now consider the horizontal sides. For
$z=x\pm\mathrm{i}\pi$, we have
\begin{align*}
\left|
1-N^{-1}x_j^{(N)}\mathrm{e}^{-z}
\right|
=
1+
N^{-1}x_j^{(N)}\mathrm{e}^{-x}
\ge
1, \quad j\in\Ibr{N}.
\end{align*}
Together with the preceding estimates on the remaining factors, this
gives, after increasing $c_N$ if necessary,
\begin{align*}
|F_N(x\pm\mathrm{i}\pi,L+\mathrm{i}\tau)|
\le
c_N
\mathrm{e}^{-\tau^2/(2t)}
\mathrm{e}^{-x^2/(2t)+C|x|},
\qquad
x\le A.
\end{align*}
The bound is integrable on
$(-\infty,A]\times\mathbb{R}$. Hence,  the integrals over the two
horizontal segments converge absolutely as $R\to\infty$.

Consequently, \eqref{rscK_N} admits the representation
\begin{align}\label{K_N-halfStrip}
\widehat{\mathcal{K}}_t^N(u,v)
=
\frac{1}{(2\pi\mathrm{i})^2t}
\int_{\Gamma}dw
\int_{\Lambda}dz\,
F_N(z,w),
\end{align}
where $\Lambda$ is the  half-strip contour
\begin{align}\label{halfStripContour}
\Lambda
\defeq
\{
A+\mathrm{i}\vartheta:
-\pi\le\vartheta\le\pi
\}
\cup
\{
x+\mathrm{i}\pi:
x\le A
\}
\cup
\{
x-\mathrm{i}\pi:
x\le A
\},
\end{align}
with orientation inherited from the positively oriented rectangular
truncations. The contours are illustrated in
Figure~\ref{fig-contours}.

\begin{figure}[h]
\centering
\begin{tikzpicture}[scale=.5,>=stealth]

% Axes
\draw[->] (-7.5,0) -- (5.2,0) node[right] {\small$\Re{(z)},\Re{(w)}$};
\draw[->] (0,-2.6) -- (0,2.6) node[above] {\small$\Im$};

\node[left] at (-7,1.8) {\small$\pi$};
\node[left] at (-7,-1.8) {\small$-\pi$};

% Half-strip contour
\draw[green!50!black,thick,->] (1.1,-1.8) -- (1.1,0.2);
\draw[green!50!black,thick,->] (1.1,0.2) -- (1.1,1.8);
\draw[green!50!black,thick,->] (1.1,1.8) -- (-2.0,1.8);
\draw[green!50!black,thick] (-2.0,1.8) -- (-7,1.8);
\draw[green!50!black,thick,->] (-7,-1.8) -- (-2.0,-1.8);
\draw[green!50!black,thick,->] (-2.0,-1.8) -- (1.1,-1.8);

\node[green!50!black] at (0.5,0.7) {\small$\Lambda$};
\node[below] at (1.1,0) {\small$A$};

% Vertical w-contour
\draw[blue!60!black,thick,->] (3.2,-2.3) -- (3.2,2.3);
\node[blue!60!black,right] at (3.2,2.1) {\small$\Gamma$};
\node[below] at (3.2,0) {\small$L$};

% Finite-N poles
\foreach \x in {-3.7,-3.0,-2.45,-1.65,-0.35}
  \fill (\x,0) circle (1.8pt);

% Finite-N contour
\draw[red!90!black,dashed,thick,->] (-4.5,-1.8) -- (1.1,-1.8);
\draw[red!90!black,dashed,thick,->] (1.1,-1.8) -- (1.1,1.8);
\draw[red!90!black,dashed,thick,->] (1.1,1.8) -- (-4.5,1.8);
\draw[red!90!black,dashed,thick,->] (-4.5,1.8) -- (-4.5,-1.8);

\node[red!90!black] at (-1.8,-1.35) {\small$\Lambda_N$};

% Left wall
\draw[orange!80!black,dash dot,thick,->]
(-5.5,1.8) -- (-5.5,-1.8);
\node[orange!80!black,above] at (-5.5,1.8)
{\small$\Lambda^{\mathrm{left}}(R)$};

\end{tikzpicture}
\caption{The fixed $w$-contour $\Gamma$, the finite $z$-contour
$\Lambda_N$, the half-strip $z$-contour $\Lambda$, and representative
finite-$N$ poles in the fundamental strip. The orange dash-dotted line
denotes $\Lambda^{\mathrm{left}}(R)$, which is sent to $-\infty$.}
\label{fig-contours}
\end{figure}
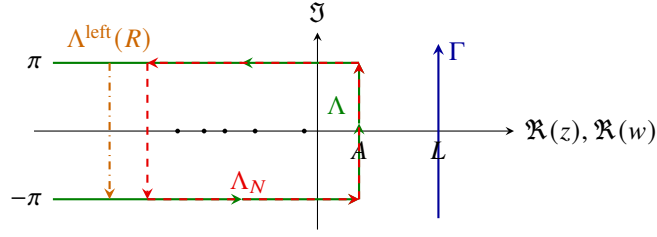

We now take $N\to\infty$ in \eqref{K_N-halfStrip}.
Note first that from \eqref{IC-conv-exp}
 we have 
$\upsilon^{(N)}
\xrightarrow{N\to\infty}
\upsilon$
in $\Upsilon$.
Hence
\begin{align*}
\varphi_{\upsilon^{(N)}}
\xrightarrow{N\to\infty}
\varphi_{\upsilon}
\qquad
\textnormal{ in }\,\LP_+,
\end{align*}
and therefore, since
$\varphi_{\upsilon}(\mathrm{e}^{-z})\ne0$ for $z\in\Lambda$,
\begin{align*}
\Phi_N(z,w)
\xrightarrow{N\to\infty}
\Phi(z,w)
\defeq
\frac{
\varphi_{\upsilon}(\mathrm{e}^{-w})
}{
\varphi_{\upsilon}(\mathrm{e}^{-z})
},
\end{align*}
pointwise on $\Lambda\times\Gamma$.

We next establish an integrable bound independent of $N$. 

Note first that 
since $\Re(w-z)\ge L-A$ on $\Gamma\times\Lambda$, we have 
\begin{align*}
  \left|\frac{1}{\mathrm e^{w-z}-1}\right|\le C,  \qquad (w,z)\in\Gamma\times\Lambda. 
\end{align*}
Moreover, 
by virtue of \eqref{GammaContourBound} we can write 
\begin{align*}
\left|
\exp\left(
\frac{w^2}{2t}
-
\frac{wv}{t}
+
\frac12w
\right)
\right|
\le
C\mathrm{e}^{-\tau^2/(2t)}, \qquad w=L+\mathrm{i}\tau, 
\end{align*}
where the constant $C>0$ is independent of $N$.

Now, consider the horizontal parts of $\Lambda$, where
$z=x\pm\mathrm{i}\pi$ and $x\le A$. 
From \eqref{IC-conv-exp}, we have
\begin{align*}
\sup_{N\ge N_0}
N^{-1}
\sum_{j=1}^{N}
x_j^{(N)}
<
\infty.
\end{align*}
Then, since
\begin{align*}
\left|
1-N^{-1}x_j^{(N)}\mathrm{e}^{-z}
\right|
=
1+
N^{-1}x_j^{(N)}\mathrm{e}^{-x}
\ge
1, \quad j\in\Ibr{N},
\end{align*}
we get
\begin{align*}
|\Phi_N(z,w)|
&\le
\prod_{j=1}^{N}
\left(
1+
N^{-1}x_j^{(N)}\mathrm{e}^{-L}
\right)
\le
\exp\left(
\mathrm{e}^{-L}
N^{-1}
\sum_{j=1}^{N}
x_j^{(N)}
\right)
\le
C.
\end{align*}
Combining the above estimates, we obtain
\begin{align*}
|F_N(x\pm\mathrm{i}\pi,L+\mathrm{i}\tau)|
\le
C
\mathrm{e}^{-\tau^2/(2t)}
\mathrm{e}^{-x^2/(2t)+C|x|}.
\end{align*}

Now consider the vertical part, where
$z=A+\mathrm{i}\vartheta$ and $|\vartheta|\le\pi$. 
Since
$N^{-1}x_1^{(N)}
\longrightarrow
x_1$ 
and $\log x_1<A$, there exists $0<\delta<1$ such that
\begin{align*}
\max_{1\le j\le N}
N^{-1}x_j^{(N)}\mathrm{e}^{-A}
=
N^{-1}x_1^{(N)}\mathrm{e}^{-A}
\le
\delta.
\end{align*}
Therefore,
\begin{align*}
\left|
1-N^{-1}x_j^{(N)}\mathrm{e}^{-z}
\right|
\ge
1-
N^{-1}x_j^{(N)}\mathrm{e}^{-A}.
\end{align*}
Using the bound $-\log(1-r)\le Cr$ for $0\le r\le\delta$, we obtain
\begin{align*}
\prod_{j=1}^{N}
\frac{1}{
\left|
1-N^{-1}x_j^{(N)}\mathrm{e}^{-z}
\right|
}
&\le
\exp\left(
C
N^{-1}
\sum_{j=1}^{N}
x_j^{(N)}
\right)
\le 
C.
\end{align*}
The numerator is bounded in the same way, and therefore
\begin{align*}
|\Phi_N(z,w)|
\le
C.
\end{align*}
Since the $z$-dependent exponential factor is bounded on the compact
vertical part, the estimates established above give
\begin{align*}
|F_N(A+\mathrm{i}\vartheta,L+\mathrm{i}\tau)|
\le
C\mathrm{e}^{-\tau^2/(2t)},
\qquad
|\vartheta|\le\pi.
\end{align*}

The preceding bounds may be chosen uniformly when $(u,v)$ ranges over
a compact subset of $\mathbb{R}^2$. Thus, for every such compact
subset, $F_N$ is dominated on $\Lambda\times\Gamma$ by an integrable
function independent of $(u,v)$ and of all sufficiently large $N$.
The dominated convergence theorem therefore yields
\begin{align*}
\widehat{\mathcal{K}}_t^N(u,v)
\longrightarrow
\mathcal{K}_t(u,v),
\qquad
N\to\infty,
\end{align*}
locally uniformly on $\mathbb{R}^2$, where
\begin{align*}
\mathcal{K}_t(u,v)
=
\frac{1}{(2\pi\mathrm{i})^2t}
\int_{\Gamma}dw
\int_{\Lambda}dz\,
\frac{1}{\mathrm{e}^{w-z}-1}
\exp\left(
\frac{w^2-z^2}{2t}
+
\frac{zu-wv}{t}
-
\frac12(z-w)
\right)
\frac{
\varphi_{\upsilon}(\mathrm{e}^{-w})
}{
\varphi_{\upsilon}(\mathrm{e}^{-z})
}.
\end{align*}

Finally, observe that the map
\begin{align*}
(\bm{x},\gamma)
\longmapsto
\sum_{i:x_i>0}\delta_{\log x_i}
\end{align*}
from $\Upsilon$ to $\mathsf{Conf}(\mathbb{R})$ is continuous (in the
vague topology). Indeed, for every compact set $\mathscr{K}\subset\mathbb{R}$,
\begin{align*}
\#\{i:\log x_i\in \mathscr{K}\}
\le
\gamma\mathrm{e}^{-\inf \mathscr{K}},
\end{align*}
and hence, only finitely many coordinates contribute on $\mathscr{K}$, locally uniformly
in $(\bm{x},\gamma)$. Since the coordinates are decreasing, these
coordinates belong locally to a fixed finite set of indices. Thus, the preceding local bound, together with coordinatewise
convergence, implies vague convergence of the associated point
configurations. Consequently, Theorem~\ref{thm-pathSpaceConv} and the
continuous mapping theorem yield convergence in distribution of the
rescaled fixed-time point configurations in
\(\mathsf{Conf}(\mathbb{R})\), and hence convergence of their Laplace
functionals for any nonnegative continuous compactly supported test
functions. On
the other hand, the locally uniform convergence
$\widehat{\mathcal{K}}_t^N\longrightarrow\mathcal{K}_t$, together with
the preceding local bounds and Hadamard's inequality, justifies passage
to the limit in the Fredholm expansions of the Laplace functionals for
compactly supported nonnegative test functions; see
\cite[Theorem~2 and Equations~(1.28)--(1.30)]{Soshnikov2000}.
Comparing the two limits shows that the Laplace functional of $\Xi_t$
is the Fredholm determinant associated with $\mathcal{K}_t$. Hence
$\Xi_t$ is a determinantal point process with correlation kernel
$\mathcal{K}_t$.

The formula for general $\theta\in\mathbb{R}$ follows from the
deterministic translation at the beginning of the proof. This completes
the proof.
\end{proof}

\subsection{Multitime correlation kernel}

In this subsection, we derive the multitime analogues of the
fixed-time kernels obtained above and prove
Theorem~\ref{thmIntro-multitimeKernel}.
We first use the Eynard--Mehta
theorem \cite{EynardMehta1998,BorodinRains}, together with a residue calculation
analogous to that in \cite[Theorem~2.2]{Johansson2004}, to obtain a
double-contour representation of the extended correlation kernel for
non-colliding Brownian bridges with arithmetic terminal points. We
then let the terminal time tend to infinity. As in the fixed-time
argument, the bridge ensembles converge in finite-dimensional
distributions to non-colliding Brownian motions with equally spaced
drifts, and the corresponding contour kernels converge to an extended
correlation kernel of the limiting process.

Write
\begin{align*}
p_{s,t}(x,y)
\defeq
p_{t-s}(x,y)
=
\frac{1}{\sqrt{2\pi(t-s)}}
\exp\left(
-\frac{(y-x)^2}{2(t-s)}
\right),
\qquad
0\le s<t.
\end{align*}

\begin{prop}\label{prop-bridgeMultitimeKernel}
Let $N\in\mathbb{N}$, $T>0$, and $\beta>0$. Consider
$N$ non-colliding Brownian bridges on $[0,T]$, starting from
$\bm{a}=(a_i)_{i=1}^N\in\mathbb{W}_N^\circ$
and ending at
$\bm{b}=(b_i)_{i=1}^N\in\mathbb{W}_N^\circ$ defined by
\begin{align*}
b_i
=
\beta\left(\frac{N+1}{2}-i\right),
\qquad
i\in\Ibr{N}.
\end{align*}
Then, an extended
correlation kernel of the bridge ensemble is given, for $s,t\in(0,T)$ and $u,v\in\mathbb{R}$, by
\begin{align}
\label{bridgeMultitimeContourKernel}
\mathcal{K}^{N,T}(s,u;t,v)
\nonumber\\={}&
-\mathbf{1}_{\{s<t\}}p_{t-s}(u,v)
\nonumber\\
&+
\frac{\beta}{(2\pi\mathrm{i})^2T\sqrt{st}}
\exp\left(
\frac{v^2}{2(T-t)}
-
\frac{u^2}{2(T-s)}
\right)
\int_{\Gamma}dw
\int_{\Lambda}dz\,
\frac{1}{
\mathrm{e}^{\beta(w-z)/T}-1
}
\nonumber\\
&\;\times
\exp\left(
\frac{T-s}{2sT}
\left(
w-\frac{T}{T-s}u
\right)^2
-
\frac{T-t}{2tT}
\left(
z-\frac{T}{T-t}v
\right)^2
+
\frac{\beta(N-1)}{2T}(z-w)
\right)
\nonumber\\
&\;\times
\prod_{j=1}^{N}
\frac{
\mathrm{e}^{\beta w/T}-\mathrm{e}^{\beta a_j/T}
}{
\mathrm{e}^{\beta z/T}-\mathrm{e}^{\beta a_j/T}
}.
\end{align}
Here $\Gamma=L+\mathrm{i}\mathbb{R}$ is oriented upward, and
$\Lambda$ is a positively oriented simple closed contour enclosing
$a_1,\ldots,a_N$, but no other poles of the $z$-integrand, and lying
strictly to the left of $\Gamma$.
\end{prop}

We have not found this explicit double-contour representation of the
extended bridge kernel in the literature. We therefore include a proof
below, combining the Eynard--Mehta theorem with the residue calculation
underlying Johansson's fixed-time formula.

\begin{proof}
Define the matrix $\mathbf{M}=[M_{ij}]_{i,j=1}^N$ by
\begin{align*}
M_{ij}
\defeq
p_T(a_i,b_j),
\qquad
i,j\in\Ibr{N}.
\end{align*}
By the Eynard--Mehta theorem \cite{EynardMehta1998,BorodinRains}, an extended
correlation kernel of the process is given by
\begin{align}\label{bridgeEynardMehtaKernel}
\mathcal{K}^{N,T}(s,u;t,v)
={}&
-\mathbf{1}_{\{s<t\}}p_{t-s}(u,v)
+
\sum_{i,j=1}^{N}
p_{T-s}(u,b_i)
\big(\mathbf{M}^{-1}\big)_{ij}
p_t(a_j,v).
\end{align}

Set
\begin{align*}
\bm{q}^{\mathsf{T}}
=
(q_i)_{1\le i\le N},
\qquad
q_i
\defeq
p_{T-s}(u,b_i),
\qquad
i\in\Ibr{N}.
\end{align*}
Using Cramer's rule, we can write
\begin{align}\label{multitimeCramer}
\sum_{i=1}^{N}
q_i
\big(\mathbf{M}^{-1}\big)_{ij}
=
\frac{
\det\left(
\mathbf{M}^{(j\to\bm{q}^{\mathsf{T}})}
\right)
}{
\det(\mathbf{M})
},
\qquad
j\in\Ibr{N},
\end{align}
where $\mathbf{M}^{(j\to\bm{q}^{\mathsf{T}})}$ is obtained
from $\mathbf{M}$ by replacing its $j$-th row by
$\bm{q}^{\mathsf{T}}$.

For each fixed $j$, let
\begin{align*}
\widehat{a}_j(\xi)
\defeq
\frac{T}{T-s}u
+
\mathrm{i}\xi
\sqrt{\frac{Ts}{T-s}},
\qquad
\xi\in\mathbb{C}.
\end{align*}
Fix $L>a_1$ and set
$\Gamma
\defeq
L+\mathrm{i}\mathbb{R}$. 
Choose
\begin{align*}
\lambda
\defeq
\left(
\frac{T}{T-s}u-L
\right)
\sqrt{\frac{T-s}{Ts}},
\end{align*}
so that
\begin{align*}
\Re\widehat{a}_j(\xi)=L,
\qquad
\xi\in\mathbb{R}+\mathrm{i}\lambda.
\end{align*}
By the Fourier representation of the heat kernel, followed by a
contour shift from $\mathbb{R}$ to
$\mathbb{R}+\mathrm{i}\lambda$, we can write
\begin{align*}
p_{T-s}(u,b_i)
={}&
\sqrt{\frac{T}{2\pi(T-s)}}
\exp\left(
-\frac{u^2}{2(T-s)}
\right)
\nonumber\\
&\quad\times
\int_{\mathbb{R}+\mathrm{i}\lambda}
\exp\left(
-\frac{\xi^2}{2}
+
\frac{\widehat{a}_j(\xi)^2}{2T}
\right)
p_T\left(
\widehat{a}_j(\xi),b_i
\right)
\,d\xi.
\end{align*}
Taking $\widehat{a}_k(\xi)=a_k$ for $k\ne j$, we have
\begin{align*}
\det\left(
\mathbf{M}^{(j\to\bm{q}^{\mathsf{T}})}
\right)
={}&
\sqrt{\frac{T}{2\pi(T-s)}}
\exp\left(
-\frac{u^2}{2(T-s)}
\right)
\nonumber\\
&\quad\times
\int_{\mathbb{R}+\mathrm{i}\lambda}
\exp\left(
-\frac{\xi^2}{2}
+
\frac{\widehat{a}_j(\xi)^2}{2T}
\right)
\det\left(
p_T(\widehat{a}_k(\xi),b_i)
\right)_{k,i=1}^{N}
\,d\xi.
\end{align*}

For the arithmetic endpoints \eqref{multitimeEndpoints}, expansion of
the heat kernels and the Vandermonde determinant gives
\begin{align}\label{detRatio}
\frac{
\det\left(
\mathbf{M}^{(j\to\bm{q}^{\mathsf{T}})}
\right)
}{
\det(\mathbf{M})
}
={}&
\sqrt{\frac{T}{2\pi(T-s)}}
\exp\left(
-\frac{u^2}{2(T-s)}
\right)
\nonumber\\
&\quad\times
\int_{\mathbb{R}+\mathrm{i}\lambda}
\exp\left(
-\frac{\xi^2}{2}
+
\frac{a_j^2}{2T}
+
\frac{\beta(N-1)}{2T}
\left(
a_j-\widehat{a}_j(\xi)
\right)
\right)
\nonumber\\
&\quad\times
\prod_{k=1, k\ne j}^{N}
\frac{
\mathrm{e}^{\beta\widehat{a}_j(\xi)/T}
-
\mathrm{e}^{\beta a_k/T}
}{
\mathrm{e}^{\beta a_j/T}
-
\mathrm{e}^{\beta a_k/T}
}
\,d\xi.
\end{align}
Making the change of variables
$w=\widehat{a}_j(\xi)$ in \eqref{detRatio}, under which
$\mathbb{R}+\mathrm{i}\lambda$ is mapped onto $\Gamma$, we obtain
\begin{align*}
\sum_{i=1}^{N}
q_i
\big(\mathbf{M}^{-1}\big)_{ij}
={}&
\frac{1}{\mathrm{i}\sqrt{2\pi s}}
\exp\left(
-\frac{u^2}{2(T-s)}
+
\frac{a_j^2}{2T}
+
\frac{\beta(N-1)}{2T}a_j
\right)
\nonumber\\
&\quad\times
\int_{\Gamma}
\exp\left(
\frac{T-s}{2sT}
\left(
w-\frac{T}{T-s}u
\right)^2
-
\frac{\beta(N-1)}{2T}w
\right)
\nonumber\\
&\quad\times
\prod_{k=1, k\ne j}^{N}
\frac{
\mathrm{e}^{\beta w/T}-\mathrm{e}^{\beta a_k/T}
}{
\mathrm{e}^{\beta a_j/T}-\mathrm{e}^{\beta a_k/T}
}
\,dw.
\end{align*}

Multiplying by $p_t(a_j,v)$, summing over $j$, and completing the
square yields
\begin{align}\label{bridgeKernelDoubleSum}
&\sum_{i,j=1}^{N}
p_{T-s}(u,b_i)
\big(\mathbf{M}^{-1}\big)_{ij}
p_t(a_j,v)
\nonumber\\
&=
\frac{1}{2\pi\mathrm{i}\sqrt{st}}
\exp\left(
\frac{v^2}{2(T-t)}
-
\frac{u^2}{2(T-s)}
\right)
\int_{\Gamma}dw\,
\exp\left(
\frac{T-s}{2sT}
\left(
w-\frac{T}{T-s}u
\right)^2
-
\frac{\beta(N-1)}{2T}w
\right)
\nonumber\\
&\quad\times
\sum_{j=1}^{N}
\exp\left(
-\frac{T-t}{2tT}
\left(
a_j-\frac{T}{T-t}v
\right)^2
+
\frac{\beta(N-1)}{2T}a_j
\right)
\prod_{k=1, k\ne j}^{N}
\frac{
\mathrm{e}^{\beta w/T}-\mathrm{e}^{\beta a_k/T}
}{
\mathrm{e}^{\beta a_j/T}-\mathrm{e}^{\beta a_k/T}
}.
\end{align}

The sum over $j$ can be written as a residue integral:
\begin{align*}
\sum_{j=1}^{N}
\exp&\left(
-\frac{T-t}{2tT}
\left(
a_j-\frac{T}{T-t}v
\right)^2
+
\frac{\beta(N-1)}{2T}a_j
\right)
\prod_{k=1, k\ne j}^{N}
\frac{
\mathrm{e}^{\beta w/T}-\mathrm{e}^{\beta a_k/T}
}{
\mathrm{e}^{\beta a_j/T}-\mathrm{e}^{\beta a_k/T}
}
\nonumber\\
&=
\frac{\beta}{2\pi\mathrm{i}T}
\int_{\Lambda}dz\,
\exp\left(
-\frac{T-t}{2tT}
\left(
z-\frac{T}{T-t}v
\right)^2
+
\frac{\beta(N-1)}{2T}z
\right)
\nonumber\\
&\quad\times
\frac{1}{
\mathrm{e}^{\beta(w-z)/T}-1
}
\prod_{k=1}^{N}
\frac{
\mathrm{e}^{\beta w/T}-\mathrm{e}^{\beta a_k/T}
}{
\mathrm{e}^{\beta z/T}-\mathrm{e}^{\beta a_k/T}
},
\end{align*}
where $\Lambda$ is as in the statement. Substitution into
\eqref{bridgeKernelDoubleSum} gives
\eqref{bridgeMultitimeContourKernel}.
\end{proof}

\begin{lem}\label{lem-driftedBmMultitimeKernel}
Let $N\in\mathbb{N}$ and
$\bm{a}=(a_i)_{i=1}^N\in\mathbb{W}_N^\circ$. Consider
the non-colliding Brownian motions started from $\bm{a}$ with
drift vector
\begin{align*}
\bm{\mu}
=
(\mu_i)_{i=1}^N,
\qquad
\mu_i
=
\frac{N+1}{2}-i,
\qquad
i\in\Ibr{N}.
\end{align*}
Then the associated space-time point process is determinantal, with an
extended correlation kernel given, for $s,t>0$ and $u,v\in\mathbb R$, by
\begin{align}\label{driftedBmMultitimeKernel}
\mathcal{K}^N(s,u;t,v)
={}&
-\mathbf{1}_{\{s<t\}}p_{t-s}(u,v)
\nonumber\\
&+
\frac{1}{(2\pi\mathrm{i})^2\sqrt{st}}
\int_{\Gamma}dw
\int_{\Lambda}dz\,
\frac{1}{\mathrm{e}^{w-z}-1}
\nonumber\\
&\quad\times
\exp\left(
\frac{(w-u)^2}{2s}
-
\frac{(z-v)^2}{2t}
+
\frac{N-1}{2}(z-w)
\right)
\prod_{j=1}^{N}
\frac{
\mathrm{e}^w-\mathrm{e}^{a_j}
}{
\mathrm{e}^z-\mathrm{e}^{a_j}
}.
\end{align}
Here $\Gamma$ and $\Lambda$ are admissible contours as in
Proposition~\ref{prop-bridgeMultitimeKernel}.
\end{lem}

\begin{proof}
We first show, by a multitime extension of the fixed-time argument in
the proof of Lemma~\ref{lem-driftedBmKernel}, that the
finite-dimensional distributions of the non-colliding Brownian
bridges converge to those of the non-colliding Brownian motions with
drift $\bm{\mu}$.

Fix
\begin{align*}
0<t_1<\cdots<t_m<\infty,
\end{align*}
and, for $T>t_m$, consider the non-colliding Brownian bridges from
$\bm{a}$ to $T\bm{\mu}$ on $[0,T]$.
By repeated application of the Karlin--McGregor formula, their joint
density at times $t_1,\ldots,t_m$ is
\begin{align*}
&\det\left(
p_{t_1}(a_i,x_j^{(1)})
\right)_{i,j=1}^{N}
\prod_{r=1}^{m-1}
\det\left(
p_{t_{r+1}-t_r}
\bigl(x_i^{(r)},x_j^{(r+1)}\bigr)
\right)_{i,j=1}^{N}
\frac{
\det\left(
p_{T-t_m}(x_i^{(m)},T\mu_j)
\right)_{i,j=1}^{N}
}{
\det\left(
p_T(a_i,T\mu_j)
\right)_{i,j=1}^{N}
}.
\end{align*}
From the Gaussian determinant computation used in the proof of
Lemma~\ref{lem-driftedBmKernel}, we have 
\begin{align*}
\frac{
\det\left(
p_{T-t_m}(x_i^{(m)},T\mu_j)
\right)_{i,j=1}^{N}
}{
\det\left(
p_T(a_i,T\mu_j)
\right)_{i,j=1}^{N}
}
\longrightarrow
\exp\left(
-\frac{t_m}{2}\lvert\bm{\mu}\rvert^2
\right)
\frac{
h_{\bm{\mu}}(\bm{x}^{(m)})
}{
h_{\bm{\mu}}(\bm{a})
},\qquad \textnormal{ as }\, T\to\infty.
\end{align*}
 Consequently, the preceding joint density converges
pointwise to the multitime density of the non-colliding Brownian
motions with drift $\bm{\mu}$. Since both the approximating
and limiting functions are probability densities, Scheff\'e's lemma
yields convergence in $L^1$. In particular, the corresponding
multitime correlation functions converge in $L^1$.

Now set $\beta=T$ in
Proposition~\ref{prop-bridgeMultitimeKernel} and let $T\to\infty$.
The pointwise convergence and domination follow as in the proof of
Lemma~\ref{lem-driftedBmKernel}, uniformly when the spatial variables
range over compact sets. The extended bridge kernels therefore
converge locally uniformly to the kernel in
\eqref{driftedBmMultitimeKernel}. Comparing the resulting
determinantal correlation functions with the preceding $L^1$ limits
proves the claim.
\end{proof}

\subsubsection{Proof of Theorem~\ref{thmIntro-multitimeKernel}}

We are now ready to establish the multitime determinantal structure and extended kernel stated in Theorem~\ref{thmIntro-multitimeKernel}.

\begin{proof}[Proof of Theorem~\ref{thmIntro-multitimeKernel}]

The argument follows the same procedure as that of
Theorem~\ref{thmIntro-fixedTimeKernel}, now applied to the extended
finite-$N$ kernel in
Lemma~\ref{lem-driftedBmMultitimeKernel}.

As in the proof of the fixed-time kernel formula, the case
$\theta\ne0$ is again obtained by deterministic translation:
\begin{align*}
\mathcal{K}^{(\theta)}(s,u;t,v)
=
\mathcal{K}^{(0)}
\left(
s,u-\frac{\theta s}{2};
t,v-\frac{\theta t}{2}
\right).
\end{align*}
It is therefore enough to consider $\theta=0$, and we write
$\mathcal{K}\defeq\mathcal{K}^{(0)}$.

Choose initial configurations
$\bm{x}^{(N)}
=\big(x_i^{(N)}\big)_{i=1}^N
\in\mathbb{W}_{N,+}^\circ$
satisfying \eqref{IC-conv-exp}, and set
\begin{align*}
\alpha_i^{(N)}
\defeq
\log x_i^{(N)},
\qquad
i\in\Ibr{N}.
\end{align*}
Let $L>A>\log x_1$ and take
$N_0\in\mathbb{N}$ large enough so that
\begin{align*}
\alpha_1^{(N)}-\log N<A,
\qquad
N\ge N_0.
\end{align*}
Applying Lemma~\ref{lem-driftedBmMultitimeKernel}, for every
$N\ge N_0$, the extended correlation kernel of the process
$\mathfrak{y}^{(N)}$ with initial condition
$\bm{\alpha}^{(N)}$ can be written, for $s,t>0$ and
$u,v\in\mathbb{R}$, as
\begin{align*}
\mathcal{K}^N(s,u;t,v)
={}&
-\mathbf{1}_{\{s<t\}}p_{t-s}(u,v)
\nonumber\\
&+
\frac{1}{(2\pi\mathrm{i})^2\sqrt{st}}
\int_{\Gamma_N}dw
\int_{\Lambda_N}dz\,
\frac{1}{\mathrm{e}^{w-z}-1}
\nonumber\\
&\quad\times
\exp\left(
\frac{(w-u)^2}{2s}
-
\frac{(z-v)^2}{2t}
+
\frac{N-1}{2}(z-w)
\right)
\prod_{j=1}^{N}
\frac{
\mathrm{e}^w-\mathrm{e}^{\alpha_j^{(N)}}
}{
\mathrm{e}^z-\mathrm{e}^{\alpha_j^{(N)}}
},
\end{align*}
where $\Gamma_N$ and $\Lambda_N$ are the contours introduced in
\eqref{Gamma_N} and \eqref{Lambda_N}.

We consider the conjugated kernel:
\begin{align*}
\widetilde{\mathcal{K}}^N(s,u;t,v)
&\defeq
\exp\left(
-\frac{u^2}{2s}
+
\frac{v^2}{2t}
\right)
\mathcal{K}^N(s,u;t,v)
\nonumber\\
&=
-\mathbf{1}_{\{s<t\}}
\mathcal{P}_{s,t}(u,v)
\nonumber\\
&\quad+
\frac{1}{(2\pi\mathrm{i})^2\sqrt{st}}
\int_{\Gamma_N}dw
\int_{\Lambda_N}dz\,
\frac{1}{\mathrm{e}^{w-z}-1}
\nonumber\\
&\qquad\times
\exp\left(
\frac{w^2}{2s}
-
\frac{wu}{s}
-
\frac{z^2}{2t}
+
\frac{zv}{t}
+
\frac{N-1}{2}(z-w)
\right)
\prod_{j=1}^{N}
\frac{
\mathrm{e}^w-\mathrm{e}^{\alpha_j^{(N)}}
}{
\mathrm{e}^z-\mathrm{e}^{\alpha_j^{(N)}}
},
\end{align*}
where
\begin{align*}
\mathcal{P}_{s,t}(u,v)
&\defeq
\exp\left(
-\frac{u^2}{2s}
+
\frac{v^2}{2t}
\right)
p_{t-s}(u,v)
=
\frac{1}{\sqrt{2\pi(t-s)}}
\exp\left(
-\frac{(tu-sv)^2}{2st(t-s)}
\right),
\qquad
0<s<t.
\end{align*}

We then apply the scaling under which the process converges and perform a further conjugation by setting
\begin{align}\label{rescaledMultitimeKernel}
\widehat{\mathcal{K}}^N(s,u;t,v)
\defeq{}&
\exp\left(
\frac{(\log N)^2}{2s}
+\frac{u\log N}{s}
-\frac{(\log N)^2}{2t}
-\frac{v\log N}{t}
\right)
\nonumber\\
&\quad\times
\widetilde{\mathcal{K}}^N
\left(
s,u+\frac{Ns}{2}+\log N;
t,v+\frac{Nt}{2}+\log N
\right).
\end{align}
Observe that the exponential factors are conjugation factors and therefore do not change the determinantal space-time correlation functions. 
Note moreover that the transition term is invariant under the combined spatial shift and
conjugation above. Indeed, writing $\ell=\log N$, we have
\begin{align*}
t\left(u+\frac{Ns}{2}+\ell\right)
-s\left(v+\frac{Nt}{2}+\ell\right)
=
tu-sv+(t-s)\ell,
\end{align*}
and thus, the conjugation factor in
\eqref{rescaledMultitimeKernel} exactly cancels out the additional
$\ell$-dependent terms:
\begin{align*}
&\exp\left(
\frac{\ell^2}{2s}
+\frac{u\ell}{s}
-\frac{\ell^2}{2t}
-\frac{v\ell}{t}
\right)
\mathcal{P}_{s,t}
\left(
u+\frac{Ns}{2}+\ell,\,
v+\frac{Nt}{2}+\ell
\right)
=
\mathcal{P}_{s,t}(u,v).
\end{align*}

Next, making the changes of variables
\begin{align*}
w_{\mathrm{old}}
=
w+\log N,
\qquad
z_{\mathrm{old}}
=
z+\log N,
\end{align*}
and using
\begin{align*}
\prod_{j=1}^{N}
\frac{
\mathrm{e}^{w+\log N}
-
\mathrm{e}^{\alpha_j^{(N)}}
}{
\mathrm{e}^{z+\log N}
-
\mathrm{e}^{\alpha_j^{(N)}}
}
=
\mathrm{e}^{N(w-z)}
\frac{
\varphi_{\upsilon^{(N)}}(\mathrm{e}^{-w})
}{
\varphi_{\upsilon^{(N)}}(\mathrm{e}^{-z})
},
\end{align*}
we obtain
\begin{align*}
\widehat{\mathcal{K}}^N(s,u;t,v)
={}&
-\mathbf{1}_{\{s<t\}}\mathcal{P}_{s,t}(u,v)
\nonumber\\
&+
\frac{1}{(2\pi\mathrm{i})^2\sqrt{st}}
\int_{\Gamma}dw
\int_{\Lambda_N}dz\,
\frac{1}{\mathrm{e}^{w-z}-1}
\nonumber\\
&\quad\times
\exp\left(
\frac{w^2}{2s}
-
\frac{wu}{s}
-
\frac{z^2}{2t}
+
\frac{zv}{t}
-
\frac12(z-w)
\right)
\frac{
\varphi_{\upsilon^{(N)}}(\mathrm{e}^{-w})
}{
\varphi_{\upsilon^{(N)}}(\mathrm{e}^{-z})
},
\end{align*}
where the contour $\Lambda_N$ is the shifted contour given in
\eqref{Lambda_N_new-contour}.

Exactly as in the proof of
Theorem~\ref{thmIntro-fixedTimeKernel}, the finite $z$-contour
$\Lambda_N$ may be deformed to the half-strip contour $\Lambda$
given in \eqref{halfStripContour}: the contribution of its left
vertical side vanishes as that side is sent to $-\infty$, and the
integrals over the horizontal segments converge. Consequently, we get
\begin{align}\label{rescaledFiniteMultitimeKernel}
\widehat{\mathcal{K}}^N(s,u;t,v)
={}&
-\mathbf{1}_{\{s<t\}}\mathcal{P}_{s,t}(u,v)
\nonumber\\
&+
\frac{1}{(2\pi\mathrm{i})^2\sqrt{st}}
\int_{\Gamma}dw
\int_{\Lambda}dz\,
\frac{1}{\mathrm{e}^{w-z}-1}
\nonumber\\
&\quad\times
\exp\left(
\frac{w^2}{2s}
-
\frac{wu}{s}
-
\frac{z^2}{2t}
+
\frac{zv}{t}
-
\frac12(z-w)
\right)
\frac{
\varphi_{\upsilon^{(N)}}(\mathrm{e}^{-w})
}{
\varphi_{\upsilon^{(N)}}(\mathrm{e}^{-z})
}.
\end{align}

The bounds established in the proof of
Theorem~\ref{thmIntro-fixedTimeKernel} remain valid, with $s$ in the
$w$-dependent Gaussian factor and $t$ in the $z$-dependent Gaussian
factor. In particular, for $w=L+\mathrm{i}\tau$,
\begin{align*}
\left|
\exp\left(
\frac{w^2}{2s}
-\frac{wu}{s}
+\frac12w
\right)
\right|
\le
C\mathrm{e}^{-\tau^2/(2s)},
\end{align*}
and, on the horizontal parts $z=x\pm\mathrm{i}\pi$,
\begin{align*}
\left|
\exp\left(
-\frac{z^2}{2t}
+\frac{zv}{t}
-\frac12z
\right)
\right|
\le
C\mathrm{e}^{-x^2/(2t)+C|x|}.
\end{align*}
The preceding estimates may be chosen uniformly when $(u,v)$ ranges
over a compact subset of $\mathbb{R}^2$. Together with the product
estimates established in the fixed-time proof, they give an
integrable majorant independent of all sufficiently large $N$ and of
$(u,v)$ in that compact subset. 
By the dominated convergence theorem, the right-hand side of
\eqref{rescaledFiniteMultitimeKernel} converges to 
\begin{align}\label{limitingMultitimeKernel-standard}
\mathcal{K}(s,u;t,v)
={}&
-\mathbf{1}_{\{s<t\}}\mathcal{P}_{s,t}(u,v)
\nonumber\\
&+
\frac{1}{(2\pi\mathrm{i})^2\sqrt{st}}
\int_{\Gamma}dw
\int_{\Lambda}dz\,
\frac{1}{\mathrm{e}^{w-z}-1}
\nonumber\\
&\quad\times
\exp\left(
\frac{w^2}{2s}
-\frac{wu}{s}
-\frac{z^2}{2t}
+\frac{zv}{t}
-\frac12(z-w)
\right)
\frac{
\varphi_{\upsilon}(\mathrm{e}^{-w})
}{
\varphi_{\upsilon}(\mathrm{e}^{-z})
},
\end{align}
locally uniformly in $(u,v)$.

Finally, by the continuity argument used in the fixed-time proof,
Theorem~\ref{thm-pathSpaceConv} implies joint vague convergence of
the rescaled configurations at every finite collection of positive
times. Hence,  their space-time Laplace functionals converge to those
of the corresponding finite-dimensional distributions of the
limiting process. On the other hand, the locally uniform convergence
of the extended kernels, together with the preceding bounds and
Hadamard's inequality, permits termwise passage to the limit in their
Fredholm expansions. The limiting Laplace functionals are therefore
the Fredholm determinants associated with
$\mathcal{K}$. Thus the limiting space-time point process is
determinantal with extended correlation kernel
$\mathcal{K}$. 
The general case follows by deterministic translation. 
This completes the proof.
\end{proof}

\begin{rmk}
At equal times, the restriction of
\eqref{limitingMultitimeKernel-Intro} is the transpose of the kernel in
Theorem~\ref{thmIntro-fixedTimeKernel}. Since transposition does not
change determinantal correlation functions, it defines the same
fixed-time determinantal point process.
\end{rmk}

We now record some consequences of
Theorem~\ref{thmIntro-fixedTimeKernel} and
Theorem~\ref{thmIntro-multitimeKernel}. We begin with the continuity of the
correlation kernels with respect to the initial state.

\begin{cor}\label{cor-kernelInitialStateContinuity}
Let $\theta\in\mathbb{R}$, and suppose that
$\upsilon_n\xrightarrow{n\to\infty}\upsilon$ in $\Upsilon$. Then,
for every $s,t>0$,
\begin{align*}
\mathcal{K}_{\upsilon_n}^{(\theta)}(s,u;t,v)
\longrightarrow
\mathcal{K}_{\upsilon}^{(\theta)}(s,u;t,v),
\qquad \textnormal{ as }\,
n\to\infty,
\end{align*}
locally uniformly in $(u,v)\in\mathbb{R}^2$, where the additional
subscript indicates the initial state. In particular, for every
$t>0$,
\begin{align*}
\mathcal{K}_{t;\upsilon_n}^{(\theta)}(u,v)
\longrightarrow
\mathcal{K}_{t;\upsilon}^{(\theta)}(u,v),
\qquad \textnormal{ as }\,
n\to\infty,
\end{align*}
locally uniformly in $(u,v)\in\mathbb{R}^2$.
\end{cor}

\begin{proof}
Write
\begin{align*}
\upsilon_n
=
\left(
\bigl(
x_i^{(n)}
\bigr)_{i\in\mathbb{N}},
\gamma_n
\right),\quad n\in\mathbb{N}, 
\qquad
\upsilon
=
\left(
\left(
x_i
\right)_{i\in\mathbb{N}},
\gamma
\right).
\end{align*}
Choose $A<L$ and $\rho\in(0,1)$ such that
$x_1\mathrm{e}^{-A}<\rho$.
Since $x_1^{(n)}\longrightarrow x_1$, there exists
$n_0\in\mathbb{N}$ such that
\begin{align*}
\sup_{n\geq n_0}
x_1^{(n)}\mathrm{e}^{-A}
\leq
\rho
<
1,
\qquad
\sup_{n\geq n_0}\gamma_n<\infty.
\end{align*}
Thus, the contours $\Gamma$ and $\Lambda$ appearing in
\eqref{limitingFixedTimeKernel-Intro} and \eqref{limitingMultitimeKernel-Intro} 
may be chosen independently of all sufficiently large $n$.
Moreover,
\begin{align*}
\varphi_{\upsilon_n}
\longrightarrow
\varphi_{\upsilon},
\qquad\textnormal{ as }\,
n\to\infty,
\end{align*}
locally uniformly on $\mathbb{C}$, and the bounds in the proofs of
Theorems~\ref{thmIntro-fixedTimeKernel} and
\ref{thmIntro-multitimeKernel} hold uniformly in $n$ and locally
uniformly in $(u,v)$. The transition term
\begin{align*}
-\mathbf{1}_{\{s<t\}}\mathcal{P}_{s,t}(u,v)
\end{align*}
is also independent of the initial state. Hence,  both conclusions
follow by applying the dominated convergence theorem.
\end{proof}

The determinantal formulas moreover yield the joint Laplace
functionals at finitely many times and, as particular cases, the
fixed-time Laplace functionals and gap probabilities.

\begin{cor}\label{cor-multitimeFredholmFormula}
Let $\theta\in\mathbb{R}$ and $\upsilon\in\Upsilon$. 
Fix $m\in\mathbb{N}$ and $0<t_1<\cdots<t_m$. Let
$f_1,\ldots,f_m$ be nonnegative, continuous, compactly supported
functions on $\mathbb{R}$. On
$\{t_1,\ldots,t_m\}\times\mathbb{R}$, set
\begin{align*}
f(t_k,u)
\defeq
f_k(u),
\qquad
g(t_k,u)
\defeq
1-\mathrm{e}^{-f_k(u)},
\qquad
k\in\llbracket m\rrbracket.
\end{align*}
Then, we have
\begin{align*}
&
\mathbb{E}_{\upsilon}
\left[
\exp\left(
-\sum_{k=1}^{m}
\int_{\mathbb{R}}
f_k(u)\,\Xi_{t_k}(\mathrm{d}u)
\right)
\right]
=
\det\left(
I-\sqrt{g}\,
\mathcal{K}^{(\theta)}
\sqrt{g}
\right)_
{L^2(\{t_1,\ldots,t_m\}\times\mathbb{R})},
\end{align*}
where the $L^2$ space is taken with respect to counting measure in
the time variable and Lebesgue measure in the spatial variable.

In particular, for every $t>0$ and every nonnegative, continuous,
compactly supported function $f$ on $\mathbb{R}$, set
\begin{align*}
g
\defeq
1-\mathrm{e}^{-f}.
\end{align*}
Then
\begin{align*}
\mathbb{E}_{\upsilon}
\left[
\exp\left(
-\int_{\mathbb{R}}
f(u)\,\Xi_t(\mathrm{d}u)
\right)
\right]
=
\det\left(
I-\sqrt{g}\,
\mathcal{K}_{t}^{(\theta)}
\sqrt{g}
\right)_{L^2(\mathbb{R})}.
\end{align*}
Moreover, for every bounded Borel set $B\subset\mathbb{R}$, we have
\begin{align*}
\mathbb{P}_{\upsilon}
\left(
\Xi_t(B)=0
\right)
=
\det\left(
I-\mathcal{K}_{t}^{(\theta)}
\right)_{L^2(B)}.
\end{align*}
Here, the correlation kernels in the Fredholm determinants are
understood as the corresponding integral operators on the indicated
$L^2$ spaces. In the gap-probability formula,
$\mathcal{K}_{t}^{(\theta)}(u,v)$ is restricted to $B\times B$.
All Fredholm determinants are understood through their absolutely
convergent Fredholm expansions.
\end{cor}

\begin{proof}
The first identity is the Fredholm expansion of the Laplace
functional of the determinantal point process on
$\{t_1,\ldots,t_m\}\times\mathbb{R}$. Absolute convergence follows
from the local bounds in the proof of
Theorem~\ref{thmIntro-multitimeKernel} and Hadamard's inequality; see
\cite[Theorem~2 and Equations~(1.28)--(1.30)]{Soshnikov2000}.

The fixed-time Laplace-functional identity is the special case
$m=1$. The gap-probability identity is the corresponding
void-probability formula for the fixed-time determinantal point
process from Theorem~\ref{thmIntro-fixedTimeKernel}. Its Fredholm
expansion is absolutely convergent by the local bounds established
in the proof of that theorem and Hadamard's inequality.
\end{proof}

Combining the preceding two corollaries shows that, at every finite
collection of positive times, the joint law of the corresponding
point configurations depends continuously on the full enhanced
initial state, where the configuration spaces are equipped with the
vague topology. In particular, the fixed-time gap probabilities are
continuous under convergence of the initial state.

We finally specialize the correlation kernels to the initial state
arising when, for every $N$, the finite-dimensional matrix process
starts from $\mathsf{Y}_N(0)=\mathbf{I}_N$.

\begin{cor}\label{cor-identityInitialKernel}
In the setting of Theorem~\ref{thmIntro-multitimeKernel}, suppose
that $\theta=0$ and
$\upsilon
=
\upsilon_{\mathrm{Id}}
\defeq
\bigl(
(0)_{i\in\mathbb{N}},
1
\bigr)$. 
Then,
\begin{align}\label{extendedKernel_id}
\mathcal{K}_{\upsilon_{\mathrm{Id}}}(s,u;t,v)
={}&
-\mathbf{1}_{\{s<t\}}\mathcal{P}_{s,t}(u,v)
\nonumber\\
&+
\frac{1}{(2\pi\mathrm{i})^2\sqrt{st}}
\int_{\Gamma}dw
\int_{\Lambda}dz\,
\frac{1}{\mathrm{e}^{w-z}-1}
\nonumber\\
&\quad\times
\exp\left(
\frac{w^2}{2s}
-\frac{wu}{s}
-\frac{z^2}{2t}
+\frac{zv}{t}
-\frac12(z-w)
+\mathrm{e}^{-z}
-\mathrm{e}^{-w}
\right),
\end{align}
where $\Gamma$ and $\Lambda$ are the contours from
Theorem~\ref{thmIntro-fixedTimeKernel}, with arbitrary finite
$A<L$.

In particular, for every $t>0$, the fixed-time point process has
correlation kernel
\begin{align*}
	\mathcal{K}_{\upsilon_{\mathrm{Id}}}(t,u;t,v)
	=
	\frac{1}{(2\pi\mathrm{i})^2t}
	\int_{\Gamma}dw
	\int_{\Lambda}dz\,
	\frac{
	\exp\left(
	\frac{w^2-z^2}{2t}
	+\frac{zv-wu}{t}
	-\frac12(z-w)
	+\mathrm{e}^{-z}
	-\mathrm{e}^{-w}
	\right)
	}{
	\mathrm{e}^{w-z}-1
	}.
\end{align*}
Equivalently, the fixed-time kernel is given by \eqref{limitingFixedTimeKernel-Intro} in
Theorem~\ref{thmIntro-fixedTimeKernel}, specialized to
$\upsilon_{\mathrm{Id}}$, which satisfies
\begin{align*}
	\mathcal{K}_{t;\upsilon_{\mathrm{Id}}}(u,v)
	=
	\mathcal{K}_{\upsilon_{\mathrm{Id}}}(t,v;t,u).
\end{align*}
\end{cor}

\begin{proof}
The extended kernel follows immediately from
Theorem~\ref{thmIntro-multitimeKernel} and the identity
$\varphi_{\upsilon_{\mathrm{Id}}}(\zeta)
=
\mathrm{e}^{-\zeta}, \;\zeta\in\mathbb{C}$.
The displayed fixed-time kernel follows by setting $s=t$.
Its transpose is exactly the fixed-time correlation kernel
\eqref{limitingFixedTimeKernel-Intro} from
Theorem~\ref{thmIntro-fixedTimeKernel}, specialized to
$\upsilon_{\mathrm{Id}}$.
\end{proof}

These formulas provide
alternative contour-integral representations of the extended and fixed-time
correlation kernels for the identity-started limiting process studied
in \cite{Ahn}. 
The following proposition identifies the extended kernel
in Corollary~\ref{cor-identityInitialKernel} with that of
\cite[Theorem~3.2(ii)]{Ahn}, up to conjugation. 
Below, we denote the Euler Gamma function by $\Gamma_\mathrm{E}(\cdot)$.

\begin{prop}
\label{prop-identity-Ahn}
Let $\mathcal{K}_{\upsilon_{\mathrm{Id}}}$ be the extended
kernel in Corollary~\ref{cor-identityInitialKernel}, and define, for $s,t>0$ and $u,v\in\mathbb{R}$,
\begin{align*}
\mathfrak{K}(s,u;t,v)
\defeq
-\mathbf{1}_{\{s<t\}}\,p_{t-s}(u,v)
+
\frac{1}{(2\pi\mathrm{i})^2}
\int_{\mathcal{C}}\mathrm{d}\xi
\int_{c+\mathrm{i}\mathbb{R}}&\mathrm{d}\eta\,
\frac{
\exp\left(t\eta^2/2-v\eta-s\xi^2/2+u\xi\right)
}{\eta-\xi}\\&\times
\frac{\Gamma_\mathrm{E}(\xi+\frac12)}{\Gamma_\mathrm{E}(\eta+\frac12)},
\end{align*}
where $c>-\frac12$, the line $c+\mathrm{i}\mathbb{R}$
is oriented upward, and $\mathcal{C}$ is the half-strip contour
\begin{align*}
\mathcal{C}
&=
\{d+\mathrm{i}y:-\varepsilon\leq y\leq\varepsilon\}
\cup
\{x+\mathrm{i}\varepsilon:x\leq d\}\cup
\{x-\mathrm{i}\varepsilon:x\leq d\},
\end{align*}
with $\varepsilon>0$ and $-\frac12<d<c$, and with orientation
inherited from positively oriented rectangular truncations. 
Then, for $s,t>0$ and $u,v\in\mathbb{R}$, we have
\begin{align}\label{KernelsEquality-id}
\mathfrak{K}(s,u;t,v)
=
\exp\left(\frac{u^2}{2s}-\frac{v^2}{2t}\right)
\mathcal{K}_{\upsilon_{\mathrm{Id}}}(s,u;t,v).
\end{align}
In particular, the two kernels yield identical
space-time correlation functions.
\end{prop}

\begin{proof}
By the identity
\begin{align*}
\frac{(tu-sv)^2}{st(t-s)}
&=
\frac{(u-v)^2}{t-s}+\frac{u^2}{s}-\frac{v^2}{t},
\qquad s<t,
\end{align*}
the stated conjugation transforms the
transition term into the standard heat kernel.

It remains to compare the double-integral terms.
Applying the conjugation introduced in \eqref{KernelsEquality-id}, the double-integral term in
\eqref{extendedKernel_id} becomes
\begin{align}\label{Qdef}
Q(s,u;t,v)
\defeq
\frac{1}{(2\pi\mathrm{i})^2\sqrt{st}}
\int_{\Gamma}\mathrm{d}w
\int_{\Lambda}\mathrm{d}z\,&
\exp\left(
\frac{(w-u)^2}{2s}-\frac{(z-v)^2}{2t}
\right)
\nonumber\\
&\times\frac{
\exp\left((w-z)/2+\mathrm{e}^{-z}-\mathrm{e}^{-w}\right)
}{
\mathrm{e}^{w-z}-1
}.
\end{align}
We have
\begin{align*}
\exp(-\mathrm{e}^{-w})
&=
\sum_{m=0}^{\infty}\frac{(-1)^m}{m!}\mathrm{e}^{-mw}.
\end{align*}
Also, since $\Re w>\Re z$, we may expand
\begin{align*}
\frac{\mathrm{e}^{(w-z)/2}}{\mathrm{e}^{w-z}-1}
&=
\sum_{n=0}^{\infty}\mathrm{e}^{-a_nw+a_nz},
\end{align*}
where
\begin{align*}
a_n=n+\frac12,\qquad n\geq0.
\end{align*}
On the integration contours, $\Re w=L$ and $\Re z\leq A<L$,
so
\begin{align*}
\sum_{n=0}^{\infty}
\left|\mathrm{e}^{-a_nw+a_nz}\right|
&=
\frac{\exp(-(L-\Re z)/2)}
{1-\exp(-(L-\Re z))},
\\
\sum_{m=0}^{\infty}
\frac{|\mathrm{e}^{-mw}|}{m!}
&=
\exp(\mathrm{e}^{-L}).
\end{align*}
Together with the Gaussian decay on $\Gamma$ and the
decay of $\exp(\mathrm{e}^{-z})$ on the horizontal
parts of $\Lambda$, these bounds allow us to integrate
the expansions term by term.

Substituting these expansions into \eqref{Qdef}, we get
\begin{align*}
Q(s,u;t,v)
=
\frac{1}{(2\pi\mathrm{i})^2\sqrt{st}}
\sum_{n,m\geq0}\frac{(-1)^m}{m!}
&\int_\Gamma
\exp\left(
\frac{(w-u)^2}{2s}-(a_n+m)w
\right)\mathrm{d}w
\\
&\times
\int_\Lambda
\exp\left(
-\frac{(z-v)^2}{2t}
+\mathrm{e}^{-z}+a_nz
\right)\mathrm{d}z.
\end{align*}
Since $a_n+m=a_{n+m}$, grouping the terms with
$k=n+m$ yields
\begin{align*}
Q(s,u;t,v)
&=
\sum_{k=0}^{\infty}
\left[
\frac{1}{2\pi\mathrm{i}\sqrt{s}}
\int_\Gamma
\exp\left(
\frac{(w-u)^2}{2s}-a_kw
\right)\mathrm{d}w
\right]
\\
&\qquad\times
\frac{1}{2\pi\mathrm{i}\sqrt{t}}
\int_\Lambda
\exp\left(
-\frac{(z-v)^2}{2t}+\mathrm{e}^{-z}
\right)
\sum_{m=0}^{k}\frac{(-1)^m}{m!}
\mathrm{e}^{(a_k-m)z}\,\mathrm{d}z.
\end{align*}
We then evaluate the integral over $\Gamma$
by completing the square and shifting the contour to
$u+sa_k+\mathrm{i}\mathbb{R}$:
\begin{align*}
\frac{1}{2\pi\mathrm{i}\sqrt{s}}
\int_{\Gamma}
\exp\left(\frac{(w-u)^2}{2s}-a_kw\right)
\mathrm{d}w
&=
\frac{1}{\sqrt{2\pi}}
\exp\left(-a_ku-\frac{s a_k^2}{2}\right).
\end{align*}
Therefore, we have
\begin{align*}
Q(s,u;t,v)
&=
\sum_{k=0}^{\infty}
\exp\left(-a_ku-\frac{s a_k^2}{2}\right)
\frac{1}{2\pi\mathrm{i}\sqrt{2\pi t}}
\int_{\Lambda}
\exp\left(-\frac{(z-v)^2}{2t}+\mathrm{e}^{-z}\right)
\\
&\hspace{60mm}\times
\sum_{m=0}^{k}\frac{(-1)^m}{m!}
\mathrm{e}^{(a_k-m)z}\,\mathrm{d}z.
\end{align*}
Now, fix $c>-\frac12$. Substituting  the Gaussian integral identity
\begin{align*}
\frac{\exp(-(z-v)^2/(2t))}{\sqrt{2\pi t}}
&=
\frac{1}{2\pi\mathrm{i}}
\int_{c+\mathrm{i}\mathbb{R}}
\exp\left(\frac{t\eta^2}{2}+\eta(z-v)\right)\mathrm{d}\eta,
\end{align*}
in the above expression, we may interchange the
$\eta$- and $z$-integrals for each fixed $k$ and $m$,
by Gaussian decay in $\Im{\eta}$ and the
decay of $\exp(\mathrm{e}^{-z})$ on the horizontal
parts of $\Lambda$.
By the change of variable $z\mapsto\mathrm{e}^{-z}$ and Hankel's
formula, with the induced orientation, we have
\begin{align*}
\frac{1}{2\pi\mathrm{i}}
\int_\Lambda \exp(bz+\mathrm{e}^{-z})\,\mathrm{d}z
&=\frac{1}{\Gamma_\mathrm{E}(b+1)},
\qquad b\in\mathbb{C}.
\end{align*}
Applying this with $b=\eta+a_k-m$, we obtain
\begin{align*}
Q(s,u;t,v)
&=
\sum_{k=0}^{\infty}
\exp\left(-a_ku-\frac{s a_k^2}{2}\right)
\frac{1}{2\pi\mathrm{i}}
\int_{c+\mathrm{i}\mathbb{R}}
\exp\left(\frac{t\eta^2}{2}-v\eta\right)
\\
&\hspace{50mm}\times
\sum_{m=0}^{k}
\frac{(-1)^m}{
m!\,\Gamma_\mathrm{E}(\eta+k+\frac32-m)
}\,\mathrm{d}\eta.
\end{align*}
The finite sum simplifies to
\begin{align*}
\sum_{m=0}^{k}
\frac{(-1)^m}{
m!\,\Gamma_\mathrm{E}(\eta+k+\frac32-m)
}
&=
\frac{(-1)^k}{
k!\,\Gamma_\mathrm{E}(\eta+\frac12)(\eta+a_k)
}.
\end{align*}
This follows from
\begin{align*}
\sum_{m=0}^{k}(-1)^m\binom{x+k}{m}
&=
(-1)^k\binom{x+k-1}{k},
\end{align*}
with $x=\eta+\frac12$, after multiplying by
$\Gamma_\mathrm{E}(x+k+1)$.

Since $|\eta+a_k|\ge c+\frac12>0$, the series inside
the following integral converges absolutely and uniformly
for $\eta\in c+\mathrm{i}\mathbb{R}$.
Moreover, Stirling's formula shows that
$\exp(t\eta^2/2-v\eta)/\Gamma_{\mathrm E}(\eta+\frac12)$
is absolutely integrable on this line.
Thus we may interchange the series and the $\eta$-integral.
We therefore arrive at
\begin{align*}
Q(s,u;t,v)
&=
\frac{1}{2\pi\mathrm{i}}
\int_{c+\mathrm{i}\mathbb{R}}
\frac{
\exp\left(t\eta^2/2-v\eta\right)
}{
\Gamma_\mathrm{E}(\eta+\frac12)
}
\sum_{k=0}^{\infty}
\frac{
(-1)^k\exp\left(-a_ku -s a_k^2/2\right)
}{
k!(\eta+a_k)
}\,\mathrm{d}\eta.
\end{align*}
To express the series as a contour integral, close
the rectangular truncations of $\mathcal{C}$ along
$\Re \xi=-M-1$, with $M\in\mathbb{N}$.  The Gamma reflection formula and the
factor $\exp(-s\xi^2/2)$ ensure that the added left
vertical integral vanishes as $M\to\infty$ and that
the horizontal integrals converge absolutely.
Since
\begin{align*}
\operatorname*{Res}_{\xi=-a_k}\Gamma_\mathrm{E}(\xi+\tfrac12)
&=\frac{(-1)^k}{k!},
\end{align*}
by the residue theorem on the contour $\mathcal{C}$
defined in the proposition, we have 
\begin{align*}
\sum_{k=0}^{\infty}
\frac{
(-1)^k\exp\left(-s a_k^2/2-a_ku\right)
}{
k!(\eta+a_k)
}
&=
\frac{1}{2\pi\mathrm{i}}
\int_{\mathcal{C}}
\frac{\exp(-s\xi^2/2+u\xi)}{\eta-\xi}
\Gamma_\mathrm{E}(\xi+\tfrac12)\,\mathrm{d}\xi.
\end{align*}
On the two integration contours, $|\eta-\xi|\geq c-d>0$.
The preceding Gamma and Gaussian bounds therefore
also justify interchanging the contour integrals.
Substituting this identity yields
\begin{align*}
Q(s,u;t,v)
&=
\frac{1}{(2\pi\mathrm{i})^2}
\int_{\mathcal{C}}\mathrm{d}\xi
\int_{c+\mathrm{i}\mathbb{R}}\mathrm{d}\eta\,
\frac{
\exp\left(t\eta^2/2-v\eta-s\xi^2/2+u\xi\right)
}{
\eta-\xi
}
\frac{\Gamma_\mathrm{E}(\xi+\frac12)}{\Gamma_\mathrm{E}(\eta+\frac12)}.
\end{align*}
This completes the proof.
\end{proof}

\section{Gibbs resampling and the ISDE from general initial conditions}

In this section, we first establish structural properties of the limiting
dynamics concerning positivity and strict ordering of the particles. 
These
properties are then used to prove the Gibbs resampling statement in
Corollary~\ref{corIntro-GibbsProperty}. We assume familiarity with the
basic terminology of Gibbsian line ensembles and refer the reader to
\cite{CorwinHammond} for background. 
Finally, we derive the ISDE
representations stated in Theorem~\ref{thmIntro-ISDE} and
Corollary~\ref{corIntro-logISDE}.

In the following, we denote the equality in distribution by $\,\distreq$.

\subsection{Positivity and strict ordering}

\begin{prop}\label{prop-positiveCoordinates}
Let $\theta\in\mathbb{R}$ and
$\upsilon=(\bm{x},\gamma)\in\Upsilon$.
Consider the process $\mathsf{X}(\sbullet)$ from
Theorem~\ref{thm-pathSpaceConv}, started from $\upsilon$. Then the
following statements hold.
\begin{enumerate}
\item[(i)]
If $\mathfrak{d}(\upsilon)=0$ and
$\mathfrak{r}(\upsilon)=r<\infty$, then for any fixed $t\ge0$, we have
\begin{align}\label{finiteRankPreservation}
\mathbb{P}_{\upsilon}
\left(
\mathsf{x}_i(t)>0
\textnormal{ for }i\in\llbracket r\rrbracket,\  \
\mathsf{x}_i(t)=0
\textnormal{ for }i>r
\right)
=
1.
\end{align}
Moreover, if $r\ge1$, then, for every fixed $t>0$,
\begin{align}\label{rParticlesDistinct}
\mathbb{P}_{\upsilon}
\left(
\mathsf{x}_1(t)>
\mathsf{x}_2(t)>
\cdots>
\mathsf{x}_r(t)>0
\right)
=
1.
\end{align}

\item[(ii)]
If either $\mathfrak{d}(\upsilon)>0$ or
$\mathfrak{r}(\upsilon)=\infty$, then, for every fixed $t>0$,
\begin{align}\label{infiniteRankPositive}
\mathbb{P}_{\upsilon}
\left(
\mathsf{x}_i(t)>0
\textnormal{ for every }i\in\mathbb{N}
\right)
=
1,
\end{align}
and
\begin{align}\label{infiniteParticlesDistinct}
\mathbb{P}_{\upsilon}
\left(
\mathsf{x}_1(t)>
\mathsf{x}_2(t)>
\cdots>0
\right)
=
1.
\end{align}
\end{enumerate}
\end{prop}

\begin{proof}
By the standard results for SDEs with Lipschitz coefficients \cite{IkedaWatanabe,RevuzYor},
the matrix SDE \eqref{GL_NBmSDE} has a unique strong solution for every
initial matrix $\mathbf{Y}_0\in\mathrm{M}_N(\mathbb{C})$. Observe that
the solution started from $\mathbf{Y}_0$ is given by
\begin{align}\label{Y_NfromG_N}
\mathsf{Y}_N(t)
=
\mathrm{e}^{\theta t}
\mathbf{Y}_0\mathsf{G}_N(t),
\quad 
t\ge0,
\end{align}
where $\mathsf{G}_N(\sbullet)$ denotes the solution of
\eqref{GL_NBmSDE} with $\theta=0$, started from the identity matrix
$\mathbf{I}_N$.

At the level of squared singular values, the scalar drift amounts to
multiplication by the same strictly positive deterministic factor at any fixed time.
Hence, it preserves positivity, the number of positive coordinates, and
multiplicities. It is therefore enough to
prove the assertions for $\theta=0$.

Suppose first that
\begin{align*}
\mathfrak{d}(\upsilon)=0,
\quad 
\mathfrak{r}(\upsilon)=r<\infty.
\end{align*}
If $r=0$, then $x_i=0$ for every $i\in\mathbb{N}$ and, since
$\mathfrak{d}(\upsilon)=0$, also $\gamma=0$. Thus
$\upsilon
=
\bigl((0)_{i\in\mathbb{N}},0\bigr)$.
By \eqref{Y_NfromG_N}, the finite-dimensional matrix process started
from the zero matrix remains identically zero. Consequently, all its squared singular values are
identically zero, and the corresponding rescaled finite-dimensional state is
$\bigl((0)_{i\in\mathbb{N}},0\bigr)$ in exponential coordinates.
These finite-dimensional states converge to $\upsilon$. Therefore,
Theorem~\ref{thm-pathSpaceConv} implies that the limiting process
$\mathsf{x}$ is identically zero almost surely.

We thus assume below that $r\ge1$. Then, we have
\begin{align*}
x_1,\ldots,x_r>0,
\quad \quad 
x_i=0,\quad i>r,
\quad \  \
\gamma=\sum_{i=1}^r x_i.
\end{align*}
For every $N\ge r$, consider the
finite-dimensional squared-singular-value process started from
\begin{align}\label{finiteRankApproximation}
x_i^{(N)}
=
\begin{cases}
Nx_i, & i\in\llbracket r\rrbracket,\\
0, & r<i\le N.
\end{cases}
\end{align}
Observe that this approximation satisfies \eqref{IC-conv-exp}.
For every $N\ge r$, let
\begin{align*}
\mathsf{D}_N
\defeq
\operatorname{diag}
\left(
\sqrt{Nx_1},\ldots,\sqrt{Nx_r},0,\ldots,0
\right),
\end{align*}
whose squared singular values are given by
\eqref{finiteRankApproximation}. Since
$\mathsf{G}_N(t)$ is almost surely invertible for every $t\ge0$, from
\eqref{Y_NfromG_N} we have
\begin{align}\label{rankId}
\operatorname{rank}\mathsf{Y}_N(t)
=
\operatorname{rank}\mathsf{D}_N
=
r,
\quad 
t\ge0,\quad \textnormal{a.s.}
\end{align}
Therefore, we have, almost surely,
\begin{align*}
\mathsf{x}_i^{(N)}(t)>0,
\quad 
i\in\llbracket r\rrbracket,
\qquad
\mathsf{x}_i^{(N)}(t)=0,
\quad 
r<i\le N,
\qquad
t\ge0.
\end{align*}
Now, set
\begin{align*}
\mathsf{z}_i^{(N)}(t)
\defeq
N^{-1}\mathsf{x}_i^{(N)}(t),
\quad 
i\in\llbracket r\rrbracket,
\quad 
t\ge0.
\end{align*}
It is straightforward to verify that the process
$\bigl(\mathsf{z}_i^{(N)}(\sbullet)\bigr)_{i\in\llbracket r\rrbracket}$
satisfies (see \cite{GraczykMalecki2013,GraczykMalecki2014,AssiotisMirsajjadi2026})
\begin{align}\label{finiteNSDE_exp}
\mathrm{d}\mathsf{z}_i^{(N)}(t)
={}&
\mathsf{z}_i^{(N)}(t)\,
\mathrm{d}\mathsf{w}_i(t)
+
\sum_{j\in\llbracket r\rrbracket\setminus\{i\}}
\frac{
\mathsf{z}_i^{(N)}(t)\mathsf{z}_j^{(N)}(t)
}{
\mathsf{z}_i^{(N)}(t)-\mathsf{z}_j^{(N)}(t)
}
\,\mathrm{d}t,
\quad 
i\in\llbracket r\rrbracket,
\end{align}
with initial condition
\begin{align*}
\bigl(\mathsf{z}_i^{(N)}(0)\bigr)_{i\in\llbracket r\rrbracket}
=
(x_i)_{i\in\llbracket r\rrbracket}.
\end{align*}
Indeed, the zero coordinates make no contribution to the interaction
drift of the positive coordinates. Since \eqref{finiteNSDE_exp} and its
initial condition do not depend on $N$, uniqueness of the
$r$-particle dynamics
(\cite{GraczykMalecki2013,GraczykMalecki2014}, more specifically,
\cite[Proposition~2.1]{AssiotisMirsajjadi2026}) shows that the law of
\begin{align*}
\bigl(
\mathsf{z}_1^{(N)}(\sbullet),
\ldots,
\mathsf{z}_r^{(N)}(\sbullet)
\bigr)
\end{align*}
does not depend on $N$. Denote a process with this common law by
$\bigl(\mathsf{z}_1,\ldots,\mathsf{z}_r\bigr)$. It follows that
\begin{align*}
\left(
N^{-1}\mathsf{x}_i^{(N)}(\sbullet)
\right)_{i\in\mathbb{N}}
\overset{\textnormal{d}}{=}
\left(
\mathsf{z}_1(\sbullet),
\ldots,
\mathsf{z}_r(\sbullet),
0,0,\ldots
\right).
\end{align*}
Theorem~\ref{thm-pathSpaceConv} therefore implies
\begin{align*}
\left(
\mathsf{x}_i(\sbullet)
\right)_{i\in\mathbb{N}}
\overset{\textnormal{d}}{=}
\left(
\mathsf{z}_1(\sbullet),
\ldots,
\mathsf{z}_r(\sbullet),
0,0,\ldots
\right).
\end{align*}
Since, for every $N\ge r$, the law of
$\bigl(\mathsf{z}_1^{(N)},\ldots,\mathsf{z}_r^{(N)}\bigr)$
is concentrated on paths taking values in $(0,\infty)^r$, the same
is true of their common law.
Hence,  the preceding identity in law proves
\eqref{finiteRankPreservation}.

Suppose next that $\mathfrak{d}(\upsilon)>0$, and set
$d\defeq\mathfrak{d}(\upsilon)$.
For every $N\in\mathbb{N}$, define
\begin{align*}
\bm{x}^{(N)}
\defeq
\left(
d+Nx_i
\right)_{i=1}^N,
\end{align*}
and let $\mathsf{x}^{(N)}$ be the process started from
$\bm{x}^{(N)}$.
Observe that
\begin{align*}
N^{-1}x_i^{(N)}
\longrightarrow
x_i, \quad i\in\mathbb{N},\qquad
N^{-1}\sum_{i=1}^{N}x_i^{(N)}
\longrightarrow
d+\sum_{i=1}^{\infty}x_i
=
\gamma.
\end{align*}
Moreover, $x_i^{(N)}\ge d$ for $i\in\Ibr{N}$.
Let
\begin{align*}
\mathsf{D}_N
\defeq
\operatorname{diag}
\left(
\sqrt{x_1^{(N)}},
\ldots,
\sqrt{x_N^{(N)}}
\right).
\end{align*}
Then, we have
\begin{align*}
\mathsf{D}_N^{\dagger}\mathsf{D}_N
\ge
d\mathbf{I}_N. 
\end{align*}
Here and below, inequalities between Hermitian matrices are understood in the Loewner order. Also,  $\bullet^\dagger$ denotes the conjugate transpose of a matrix. 
Now, couple the two matrix processes by using the same Brownian flow
$\mathsf{G}_N$.
For every $t\ge0$, we have
\begin{align*}
\bigl(
\mathsf{D}_N\mathsf{G}_N(t/4)
\bigr)^{\dagger}
\bigl(
\mathsf{D}_N\mathsf{G}_N(t/4)
\bigr)
&=
\mathsf{G}_N(t/4)^{\dagger}
\mathsf{D}_N^{\dagger}\mathsf{D}_N
\mathsf{G}_N(t/4)
\\
&\ge
d\mathsf{G}_N(t/4)^{\dagger}\mathsf{G}_N(t/4).
\end{align*}
Multiplying the ordered eigenvalues on both sides by the common
positive factor $\mathrm{e}^{-Nt/2}$ and using the monotonicity of
ordered eigenvalues under positive-semidefinite matrix inequalities,
we obtain
\begin{align}\label{finiteNIdentityComparison}
\mathsf{x}_i^{(N)}(t)
\ge
d\mathsf{x}_{i;\mathrm{Id}}^{(N)}(t),
\qquad
i\in\Ibr{N},
\quad
\textnormal{a.s.},
\end{align}
where $\mathsf{x}_{\mathrm{Id}}^{(N)}$ denotes the
process $\mathsf{x}^{(N)}$ started from
$\bm{1}_N\defeq(1,\ldots,1)$.

Observe that the initial conditions $\bm{1}_N$ converge, in the
sense of \eqref{IC-conv-exp}, to
\begin{align*}
\upsilon_{\mathrm{Id}}
\defeq
\bigl((0)_{i\in\mathbb{N}},1\bigr).
\end{align*}
Applying Theorem~\ref{thm-pathSpaceConv} and denoting the corresponding
limit process arising from $\mathsf{x}_{\mathrm{Id}}^{(N)}$ by
$\mathsf{x}_{\mathrm{Id}}$, we have
\begin{align}\label{infiniteIdentityComparison}
d\mathsf{x}_{i;\mathrm{Id}}(t)
\le_{\mathrm{st}}
\mathsf{x}_i(t),
\quad 
i\in\mathbb{N},
\quad 
t\ge0,
\end{align}
where $\le_{\mathrm{st}}$ denotes the usual stochastic order; see
\cite[Theorem~1.A.3(c)]{ShakedShanthikumar2007}.

In logarithmic coordinates,
$\mathsf{x}_{\mathrm{Id}}(\sbullet)$ is identified in
\cite{AssiotisMirsajjadi2026} with the line ensemble constructed in
\cite{Ahn}. In particular, since by virtue of \cite[Theorem~1.3]{Ahn}, 
\begin{align*}
\mathbb{P}_{\upsilon_{\mathrm{Id}}}
\left(
\mathsf{x}_{i;\mathrm{Id}}(t)>0
\right)
=
1,
\quad
i\in\mathbb{N},\quad t>0,
\end{align*}
and $d>0$, we conclude from \eqref{infiniteIdentityComparison} that  
\begin{align*}
\mathbb{P}_{\upsilon}
\left(
\mathsf{x}_i(t)>0
\right)
=
1,
\quad
i\in\mathbb{N},\quad t>0.
\end{align*}
Taking the countable intersection over $i\in\mathbb{N}$, we obtain 
\eqref{infiniteRankPositive} as desired.

It remains to consider the case
$\mathfrak{d}(\upsilon)=0$ and
$\mathfrak{r}(\upsilon)=\infty$. Fix $k\in\mathbb{N}$ and define the
rank-$k$ truncation
\begin{align*}
\upsilon^{[k]}
\defeq
\left(
\left(
x_1,\ldots,x_k,0,0,\ldots
\right),
\sum_{i=1}^{k}x_i
\right).
\end{align*}
Denote by $\mathsf{x}^{[k]}$ the limiting process started from
$\upsilon^{[k]}$. By the finite-rank part of the proof,
\begin{align}\label{truncatedRankPositivity}
\mathbb{P}_{\upsilon^{[k]}}
\left(
\mathsf{x}_k^{[k]}(t)>0
\right)
=
1,
\qquad
t\ge0.
\end{align}
For $N\ge k$, define the initial conditions
\begin{align}\label{IC-truncated}
\bm{x}^{(N)}
&\defeq
\left(
Nx_1,\ldots,Nx_N
\right),
\\
\bm{x}^{(N),[k]}
&\defeq
\left(
Nx_1,\ldots,Nx_k,0,\ldots,0
\right),
\end{align}
and denote the corresponding finite-dimensional processes by
$\mathsf{x}^{(N)}$ and $\mathsf{x}^{(N),[k]}$, respectively. Let
\begin{align*}
\mathsf{D}_N
&\defeq
\operatorname{diag}
\left(
\sqrt{Nx_1},\ldots,\sqrt{Nx_N}
\right),
\\
\mathsf{D}_N^{[k]}
&\defeq
\operatorname{diag}
\left(
\sqrt{Nx_1},\ldots,\sqrt{Nx_k},0,\ldots,0
\right).
\end{align*}
Then, we have
\begin{align*}
\mathsf{D}_N^{\dagger}\mathsf{D}_N
\ge
\bigl(
\mathsf{D}_N^{[k]}
\bigr)^{\dagger}
\mathsf{D}_N^{[k]}.
\end{align*}
Coupling the two matrix processes by the same Brownian flow
$\mathsf{G}_N$, as above, yields
\begin{align*}
\mathsf{x}_k^{(N),[k]}(t)
\le
\mathsf{x}_k^{(N)}(t),
\quad
t\ge0,\quad \textnormal{a.s.}.
\end{align*}
Observe that the initial conditions in \eqref{IC-truncated} satisfy
\eqref{IC-conv-exp}, with respective limits
$\upsilon^{[k]}$ and $\upsilon$.
Dividing the preceding inequality by
$N$, applying Theorem~\ref{thm-pathSpaceConv}, and again using the closure of the usual
stochastic order under weak convergence  \cite[Theorem~1.A.3(c)]{ShakedShanthikumar2007}, we obtain,  
for every fixed $t\ge0$,
\begin{align*}
\mathsf{x}_k^{[k]}(t)
\le_{\mathrm{st}}
\mathsf{x}_k(t).
\end{align*}
Then, from \eqref{truncatedRankPositivity} and the preceding stochastic
domination, 
it follows that
\begin{align*}
\mathbb{P}_{\upsilon}
\left(
\mathsf{x}_k(t)>0
\right)
=
1,
\qquad
t\ge0.
\end{align*}
Since $k\in\mathbb{N}$ is arbitrary, taking the countable intersection
over $k$ implies  \eqref{infiniteRankPositive}.

Finally, by Theorem~\ref{thmIntro-fixedTimeKernel}, for every fixed
$t>0$, the logarithms of the positive coordinates form a determinantal
point process on $\mathbb{R}$. Its second factorial moment measure is
absolutely continuous with respect to Lebesgue measure on
$\mathbb{R}^2$ and therefore assigns zero mass to the diagonal.
Consequently, almost surely, no two positive coordinates coincide at
time $t$.
Together with \eqref{finiteRankPreservation} and
\eqref{infiniteRankPositive}, this proves 
\eqref{rParticlesDistinct} and
\eqref{infiniteParticlesDistinct}, and completes the proof.
\end{proof}

\subsection{Proofs of the Gibbs resampling and ISDE results}

\begin{proof}[Proof of Corollary~\ref{corIntro-GibbsProperty}]
If $\mathcal{I}(\upsilon)=\emptyset$, there is nothing to prove.
Fix
$\epsilon>0$. By Proposition~\ref{prop-positiveCoordinates}, almost surely,
the active coordinates are strictly positive and strictly ordered at time
$\epsilon$.

By the Markov property of the process $\mathsf{X}(\sbullet)$, conditionally on the
full state 
$\mathsf{X}(\epsilon)$, 
the shifted process
$\bigl(
\mathsf{X}(\epsilon+t)
\bigr)_{t\ge0}$
has the law of the same dynamics started from $\mathsf{X}(\epsilon)$.

If $\mathcal{I}(\upsilon)=\llbracket r\rrbracket$ for some $r\in\mathbb{N}$, then, by the
finite-rank identification in the proof of
Proposition~\ref{prop-positiveCoordinates}, the active coordinates of the
restarted process form the corresponding $r$-particle diffusion started
from a strictly ordered configuration. Hence
\cite[Proposition~2.13]{AssiotisMirsajjadi2026} gives the Brownian Gibbs
property on $[\epsilon,\infty)$. 

If
$\mathcal{I}(\upsilon)=\mathbb N$, then
$\bigl(
\mathsf{x}_i(\epsilon)
\bigr)_{i\in\mathcal{I}(\upsilon)}$ has strictly positive and strictly ordered particle
coordinates, and
\cite[Proposition~2.16]{AssiotisMirsajjadi2026} applies directly.
Consequently,
$\bigl(
\mathfrak{x}_i(t)
\bigr)_{i\in\mathcal{I}(\upsilon), t\ge\epsilon}$
satisfies the Brownian Gibbs property under
$\mathbb{P}_{\upsilon}$. 
Since $\epsilon>0$ was arbitrary, the active line ensemble satisfies the
Brownian Gibbs property on every
$[\epsilon,\infty)$, $\epsilon>0$, as desired.

Finally, Gibbs property implies that the active paths do not
intersect on any compact positive-time interval almost surely. Applying
this to
$[1/n,n],
\,
n\in\mathbb N$, 
and taking the countable intersection of the resulting probability-one
events yields \eqref{nonCollision}.
\end{proof}

\begin{rmk}
The Gibbs resampling property established here for general initial conditions provides an alternative proof of the
non-collision result in
\cite{AssiotisMirsajjadi2026}; see  \cite[Remark~2.6]{AssiotisMirsajjadi2026}.
\end{rmk}

\begin{proof}[Proof of Theorem~\ref{thmIntro-ISDE}]
If $\mathcal{I}(\upsilon)=\emptyset$, there is nothing to prove.

Suppose first that
$\mathcal{I}(\upsilon)=\llbracket r\rrbracket$ for some $r\in\mathbb{N}$.
By the finite-rank identification in the proof of
Proposition~\ref{prop-positiveCoordinates}, the active coordinates have
the law of the corresponding $r$-particle squared-singular-value
diffusion. Hence,  the conclusion follows from the finite-dimensional theory
developed in \cite{GraczykMalecki2013,GraczykMalecki2014}, or directly from
\cite[Proposition~2.1]{AssiotisMirsajjadi2026}.
It therefore remains to consider the case
$\mathcal{I}(\upsilon)=\mathbb{N}$.
Let
$\bigl(\mathcal{F}_t\bigr)_{t\ge0}$
denote the completed, right-continuous natural filtration of
$(\mathsf{X}(t))_{t\ge0}$.

Let $\epsilon>0$. By Proposition~\ref{prop-positiveCoordinates},
almost surely,
\begin{align*}
\mathsf{x}_1(\epsilon)>
\mathsf{x}_2(\epsilon)>
\cdots>0.
\end{align*}
By the Markov property of the enhanced process $\mathsf{X}(\sbullet)$,
conditionally on $\mathsf{X}(\epsilon)$, the shifted process
$\bigl(\mathsf{X}(\epsilon+t)\bigr)_{t\ge0}$ has the law of the same
dynamics started from $\mathsf{X}(\epsilon)$. Hence, applying
\cite[Proposition~2.20]{AssiotisMirsajjadi2026} conditionally on
$\mathsf{X}(\epsilon)$, the active coordinates satisfy the
corresponding ISDE on $[\epsilon,\infty)$ almost surely, with initial
values $(\mathsf{x}_i(\epsilon))_{i\in\mathbb{N}}$.  Therefore, the process 
$\big(\mathsf{m}_i^{(\epsilon)}(\sbullet)\big)_{i\in\mathbb{N}}$
defined as
\begin{align}
\label{localMartingale-epsilon}
\mathsf{m}_i^{(\epsilon)}(t)
\defeq{}&
\mathsf{x}_i(t)-\mathsf{x}_i(\epsilon)
-\frac{\theta}{2}
\int_\epsilon^t
\mathsf{x}_i(s)\,\mathrm{d}s
-
\int_\epsilon^t
\sum_{j\in\mathbb{N}\setminus\{i\}}
\frac{
\mathsf{x}_i(s)\mathsf{x}_j(s)
}{
\mathsf{x}_i(s)-\mathsf{x}_j(s)
}
\,\mathrm{d}s,
\qquad t\ge\epsilon,
\end{align}
is a family of continuous local martingales on $[\epsilon,\infty)$, with respect
to the filtration
$\bigl(\mathcal{F}_t\bigr)_{t\ge\epsilon}$, and its quadratic covariation is given by
\begin{align}
\left\langle
\mathsf{m}_i^{(\epsilon)},
\mathsf{m}_j^{(\epsilon)}
\right\rangle_t
=
\mathbf{1}_{\{i=j\}}
\int_\epsilon^t
\mathsf{x}_i(s)^2\,\mathrm{d}s,
\qquad
t\ge\epsilon.
\label{localMartingaleBracket}
\end{align}

It is enough to consider $\epsilon\in\mathbb{Q}_{>0}$. If
$0<\epsilon'<\epsilon$ are rational, then
\eqref{localMartingale-epsilon} gives
\begin{align*}
\mathsf{m}_i^{(\epsilon')}(t)
-
\mathsf{m}_i^{(\epsilon')}(\epsilon)
=
\mathsf{m}_i^{(\epsilon)}(t),
\quad
t\ge\epsilon, \quad i\in\mathbb{N}.
\end{align*}
Thus, the local martingale increments agree on overlapping
positive-time intervals. 

Note that, by virtue of Corollary~\ref{corIntro-GibbsProperty}, almost surely,
\begin{align*}
\mathsf{x}_i(s)>0,
\qquad
i\in\mathbb{N},\quad s>0.
\end{align*}
Then, for every $\epsilon>0$ and $i\in\mathbb{N}$, define
\begin{align}\label{BrownianIncrements}
\mathsf{w}_i^{(\epsilon)}(t)
\defeq
\int_\epsilon^t
\frac{1}{\mathsf{x}_i(s)}
\,\mathrm{d}\mathsf{m}_i^{(\epsilon)}(s),
\qquad t\ge\epsilon.
\end{align}
By \eqref{localMartingaleBracket}, for every $t\geq\epsilon$ we have
\begin{align*}
\left\langle
\mathsf{w}_i^{(\epsilon)},
\mathsf{w}_j^{(\epsilon)}
\right\rangle_t
=
\mathbf{1}_{\{i=j\}}(t-\epsilon).
\end{align*}
Therefore, using the L\'evy's characterization theorem, we conclude that for every $\epsilon>0$,
each finite subfamily of
 $\big(
\mathsf{w}_i^{(\epsilon)}(\epsilon+t)
\big)_{i\in\mathbb{N},\,t\geq0}$ 
is a standard multidimensional Brownian motion with respect to the
shifted filtration
$\bigl(\mathcal{F}_{\epsilon+t}\bigr)_{t\geq0}$.

Observe that for $0<\delta<\epsilon<t$, by additivity of stochastic integrals we have
\begin{align*}
\mathsf{w}_i^{(\delta)}(t)
-
\mathsf{w}_i^{(\epsilon)}(t)
=
\mathsf{w}_i^{(\delta)}(\epsilon).
\end{align*}
Since the right-hand side is a Brownian increment over an interval of
length $\epsilon-\delta$,
\begin{align*}
\mathbb{E}_{\upsilon}
\left[
\left|
\mathsf{w}_i^{(\delta)}(t)
-
\mathsf{w}_i^{(\epsilon)}(t)
\right|^4
\right]
=
3|\epsilon-\delta|^2.
\end{align*}
Thus, for every fixed $i\in\mathbb{N}$ and $t>0$,
$\mathsf{w}_i^{(\epsilon)}(t)$ is Cauchy in $L^4$ as
$\epsilon\downarrow0$ through $\mathbb{Q}_{>0}$. We therefore define,
for each $t>0$ and $i\in\mathbb{N}$
\begin{align}\label{BrownianExtensionAtZero}
\mathsf{w}_i(t)
\defeq
\lim_{\substack{\epsilon\downarrow0\\
\epsilon\in\mathbb{Q}_{>0}}}
\mathsf{w}_i^{(\epsilon)}(t)
\end{align}
as an $L^4$ limit, and set
$\mathsf{w}_i(0)=0$.

For $0<s<t$ and $0<\epsilon<s$, we have
\begin{align*}
\mathsf{w}_i^{(\epsilon)}(t)
-
\mathsf{w}_i^{(\epsilon)}(s)
=
\int_s^t
\frac{1}{\mathsf{x}_i(r)}
\,\mathrm{d}\mathsf{m}_i^{(\epsilon)}(r).
\end{align*}
The right-hand side is a Brownian increment over an interval of length
$t-s$. Letting $\epsilon\downarrow0$ in $L^4$ then gives
\begin{align*}
\mathbb{E}_{\upsilon}
\left[
\left|
\mathsf{w}_i(t)-\mathsf{w}_i(s)
\right|^4
\right]
=
3|t-s|^2.
\end{align*}
For $s=0$, since
\begin{align*}
\mathbb{E}_{\upsilon}
\left[
\big|
\mathsf{w}_i^{(\epsilon)}(t)
\big|^4
\right]
=
3(t-\epsilon)^2,
\end{align*}
the $L^4$ convergence in
\eqref{BrownianExtensionAtZero}, together with
$\mathsf{w}_i(0)=0$, implies the same identity.
Kolmogorov's continuity theorem therefore yields continuous
modifications, which, by countability, may be chosen simultaneously
for all $i\in\mathbb{N}$. We work with these modifications henceforth.

It remains to verify the martingale property of the limiting
processes $(\mathsf{w}_i)_{i\in\mathbb{N}}$ and to identify their
quadratic covariations.  Note that since
each
$\mathsf{w}_i^{(\epsilon)}(t)$ is $\mathcal{F}_t$-measurable,
$\mathsf{w}_i(t)$ is also $\mathcal{F}_t$-measurable. Moreover, for
$0<s<t$ and every rational $\epsilon\in(0,s)$,
\begin{align*}
\mathbb{E}_{\upsilon}
\left[
\mathsf{w}_i^{(\epsilon)}(t)
\,\big|\,
\mathcal{F}_s
\right]
=
\mathsf{w}_i^{(\epsilon)}(s).
\end{align*}
Letting $\epsilon\downarrow0$ in $L^1$ shows that
\begin{align*}
\mathbb{E}_{\upsilon}
\left[
\mathsf{w}_i(t)
\,\middle|\,
\mathcal{F}_s
\right]
=
\mathsf{w}_i(s).
\end{align*}
By right-continuity of the filtration, the identity extends to $s=0$.
Thus each $\mathsf{w}_i$ is a continuous
$(\mathcal{F}_t)_{t\geq0}$-martingale.

Similarly, for every $\epsilon>0$ and $i,j\in\mathbb{N}$,
\begin{align*}
\mathsf{w}_i^{(\epsilon)}(t)
\mathsf{w}_j^{(\epsilon)}(t)
-
\mathbf{1}_{\{i=j\}}(t-\epsilon),
\qquad t\geq\epsilon,
\end{align*}
is a martingale. Since the $L^4$ convergence of the factors implies
$L^2$ convergence of their products, passing to the limit shows that
\begin{align*}
\mathsf{w}_i(t)\mathsf{w}_j(t)
-
\mathbf{1}_{\{i=j\}}t,
\qquad t\geq0,
\end{align*}
is a martingale. Consequently,
\begin{align*}
\left\langle
\mathsf{w}_i,
\mathsf{w}_j
\right\rangle_t
=
\mathbf{1}_{\{i=j\}}t, \quad t\ge 0.
\end{align*}
Hence, by L\'evy's characterization theorem, 
$(\mathsf{w}_i)_{i\in\mathbb{N}}$ is a family of independent standard
Brownian motions on $\mathbb{R}_+$.

Returning to the definition of
$\mathsf{m}_i^{(\epsilon)}$, we conclude that, almost surely, for every
$i\in\mathbb{N}$ and $0<\epsilon<t$,
\begin{align}\label{ISDE-positiveTime-exp}
\mathsf{x}_i(t)
={}&
\mathsf{x}_i(\epsilon)
+
\int_\epsilon^t
\mathsf{x}_i(s)\,\mathrm{d}\mathsf{w}_i(s)
+
\frac{\theta}{2}
\int_\epsilon^t
\mathsf{x}_i(s)\,\mathrm{d}s
+
\int_\epsilon^t
\sum_{j\in\mathbb{N}\setminus\{i\}}
\frac{
\mathsf{x}_i(s)\mathsf{x}_j(s)
}{
\mathsf{x}_i(s)-\mathsf{x}_j(s)
}
\,\mathrm{d}s.
\end{align}

Finally, let $\epsilon\downarrow0$. Since $\mathsf{X}(\sbullet)$ has continuous
sample paths in $\Upsilon$,
$\mathsf{x}_i(\epsilon)\longrightarrow x_i$ almost surely. Moreover, continuity of $\mathsf{x}_i$ implies
\[
\int_0^t\mathsf{x}_i(s)^2\,\mathrm{d}s<\infty
\]
almost surely. Hence,  the stochastic integral
$\int_0^t\mathsf{x}_i(s)\,\mathrm{d}\mathsf{w}_i(s)$ is well defined,
and its continuity in the time parameter implies
\[
\int_\epsilon^t\mathsf{x}_i(s)\,\mathrm{d}\mathsf{w}_i(s)
\longrightarrow
\int_0^t\mathsf{x}_i(s)\,\mathrm{d}\mathsf{w}_i(s)
\]
almost surely as $\epsilon\downarrow0$. The ordinary integral converges
by continuity of $\mathsf{x}_i$.

Rearranging \eqref{ISDE-positiveTime-exp}, we obtain
\begin{align*}
\int_\epsilon^t
\sum_{j\in\mathbb{N}\setminus\{i\}}
\frac{\mathsf{x}_i(s)\mathsf{x}_j(s)}
{\mathsf{x}_i(s)-\mathsf{x}_j(s)}
\,\mathrm{d}s
=
\mathsf{x}_i(t)-\mathsf{x}_i(\epsilon)
-\int_\epsilon^t
\mathsf{x}_i(s)\,\mathrm{d}\mathsf{w}_i(s)
-\frac{\theta}{2}
\int_\epsilon^t
\mathsf{x}_i(s)\,\mathrm{d}s.
\end{align*}
The right-hand side has an almost surely finite limit as
$\epsilon\downarrow0$. Hence,  the interaction integral admits the
improper limit appearing in \eqref{ISDE-exponential}, and passing to
the limit in \eqref{ISDE-positiveTime-exp} proves
\eqref{ISDE-exponential}.
\end{proof}

\begin{proof}[Proof of Corollary~\ref{corIntro-logISDE}]
Fix $i\in\mathcal{I}(\upsilon)$ and $0<\epsilon<t$. 
By Corollary~\ref{corIntro-GibbsProperty}, the active coordinates are
strictly positive on $[\epsilon,t]$ almost surely. 
Restricting \eqref{ISDE-exponential} to $[\epsilon,t]$
and applying It\^o's formula to
$\mathfrak{x}_i=\log\mathsf{x}_i$, we obtain 
\begin{align*}
\mathrm{d}\mathfrak{x}_i(s)
={}&
\mathrm{d}\mathsf{w}_i(s)
+
\frac{\theta-1}{2}\,\mathrm{d}s
+
\sum_{j\in\mathcal{I}(\upsilon)\setminus\{i\}}
\frac{1}
{\mathrm{e}^{\mathfrak{x}_i(s)-\mathfrak{x}_j(s)}-1}
\,\mathrm{d}s.
\end{align*}
Integrating from $\epsilon$ to $t$ yields
\eqref{logISDE-positiveTime}.

If all active logarithmic initial coordinates are finite, then, arguing
as in the proof of Theorem~\ref{thmIntro-ISDE} and letting
$\epsilon\downarrow0$, the interaction integral admits an almost surely
finite improper limit. Hence,  \eqref{logISDE-timeZero} follows.

\end{proof}

% =================================================
% Bibliography
% =================================================

\bibliographystyle{acm}
\bibliography{References}

\bigskip

\noindent
\footnotesize\textsc{School of Mathematics, University of Edinburgh\\
%James Clerk Maxwell Building, 
Edinburgh, %EH9 3FD, 
United Kingdom}

\noindent
\footnotesize{Email:}
\href{mailto:zmirsajj@ed.ac.uk}
{zmirsajj@ed.ac.uk}

\end{document}